\documentclass[matann,final]{svjour}

\smartqed

\usepackage{graphics}
\usepackage{graphicx,graphics}
\usepackage{amsmath}
\usepackage{amssymb}
\usepackage{mathptmx}
\usepackage{color}

\newcommand{\A}{\mathcal{A}}

\newcommand{\F}{\mathcal{F}}
\newcommand{\HH}{\mathcal{H}}
\newcommand{\K}{\mathcal{K}}

\newcommand{\N}{\mathbb{N}}

\newcommand{\OO}{\mathcal{O}}
\newcommand{\PP}{\mathcal{P}}
\newcommand{\Q}{\textbf{Q}}
\newcommand{\R}{\mathbb{R}}
\newcommand{\RR}{\mathcal{R}}

\newcommand{\X}{\mathbb{X}}

\newcommand{\dom}{\mathrm{dom}}
\newcommand{\Lip}{\mathrm{Lip}}

\journalname{Mathematische Annalen}

\begin{document}

\title{A general theorem of existence of quasi absolutely minimal
  Lipschitz extensions}

\author{Matthew J. Hirn \and Erwan Y. Le Gruyer}

\institute{\textsc{Matthew J. Hirn} \at \'{E}cole normale sup\'{e}rieure\\
D\'{e}partement d'Informatique\\
45 rue d'Ulm\\
75005 Paris, France\\
\email{matthew.hirn@ens.fr}\and
Erwan Y. Le Gruyer \at Institut National des Sciences Appliqu\'{e}es
de Rennes \& IRMAR\\
20, Avenue des Buttes de Co\"{e}smes\\
CS 70839 F - 35708 Rennes Cedex 7 , France\\
\email{Erwan.Le-Gruyer@insa-rennes.fr}}

\maketitle

\begin{abstract}
In this paper we consider a wide class of generalized Lipschitz
extension problems and the corresponding problem of finding absolutely
minimal Lipschitz extensions. We prove that if a minimal Lipschitz
extension exists, then under certain other mild conditions, a quasi
absolutely minimal Lipschitz extension must exist as well. Here we use
the qualifier ``quasi'' to indicate that the extending function in
question nearly satisfies the conditions of being an absolutely
minimal Lipschitz extension, up to several factors that can be made
arbitrarily small.
\end{abstract}

\section{Introduction} \label{sec: introduction}

In this paper we attempt to generalize Aronsson's result on absolutely minimal
Lipschitz extensions for scalar valued functions to a more general setting
that includes a wide class of functions. The main result
is the existence of a ``quasi-AMLE,'' which intuitively is a function
that nearly satisfies the conditions of absolutely minimal Lipschitz extensions.

Let $E \subset \R^d$ and $f:E \rightarrow \R$ be Lipschitz continuous,
so that
\begin{equation*}
\Lip(f;E) \triangleq \sup_{\substack{x,y \in E \\ x \neq y}}
\frac{|f(x)-f(y)|}{\|x-y\|} < \infty.
\end{equation*}
The original Lipschitz extension problem then asks the following
question: is it possible to extend $f$ to a function $F: \R^d \rightarrow
\R$ such that
\begin{align*}
&F(x) = f(x), \text{ for all } x \in E, \\
&\Lip(F;\R^d) = \Lip(f;E).
\end{align*}
By the work of McShane \cite{mcshane:extensionRangeFcns1934} and
Whitney \cite{whitney:analyticExtensions1934} in $1934$, it is known that
such an $F$ exists and that two extensions can be written explicitly: 
\begin{align}
\Psi(x) &\triangleq \inf_{y \in E} (f(y) +
\Lip(f;E)\|x-y\|), \label{eqn: max lip ext} \\
\Lambda(x) &\triangleq \sup_{y \in E} (f(y) -
\Lip(f;E)\|x-y\|). \label{eqn: min lip ext}
\end{align}
In fact, the two extensions $\Psi$ and $\Lambda$ are extremal, so that
if $F$ is an arbitrary minimal Lipschitz extension of $f$, then $\Lambda \leq
F \leq \Psi$. Thus, unless $\Lambda \equiv \Psi$, the extension $F$ is
not unique, and so one can search for an extending function $F$ that
satisfies additional properties. 

In a series of papers in the $1960's$
\cite{aronsson:minSupFI1965,aronsson:minSupFII1966,aronsson:AMLE1967},
Aronsson proposed the
notion of an absolutely minimal Lipschitz extension (AMLE), which is
essentially the ``locally best'' Lipschitz
extension. His original motivation for the concept was in conjunction
with the infinity Laplacian and infinity harmonic functions. We first
define the property of absolute minimality
independently of the notion of an extension. A function $u: D
\rightarrow \R$, $D \subset \R^d$, is absolutely minimal if 
\begin{equation} \label{eqn: amle aronsson def}
\Lip(u;V) = \Lip(u;\partial V), \quad \text{for all open } V \subset\subset D,
\end{equation}
where $\partial V$ denotes the boundary of $V$, $V \subset\subset
D$ means that $\overline{V}$ is compact in $D$, and $\overline{V}$ is
the closure of $V$. A function $U$ is an AMLE for $f:E \rightarrow \R$
if $U$ is a Lipschitz extension of $f$, and furthermore, if it is also
absolutely minimal on $\R^d \setminus E$. That is:
\begin{align*}
&U(x) = f(x), \text{ for all } x \in E, \\
&\Lip(U;\R^d) = \Lip(f;E), \\
&\Lip(U;V) = \Lip(U;\partial V), \text{ for all open } V
\subset\subset \R^d \setminus E.
\end{align*}
The existence of an AMLE was proved by Aronsson, and
in $1993$ Jensen \cite{jensen:uniqueAMLE1993} proved that AMLEs are
unique under certain conditions (see also
\cite{barles:existNonLinearElliptic2001,armstrong:easyProofJensen2010}).

Since the work of Aronsson, there has been much research devoted to
the study of AMLEs and problems related to them. For a
discussion on many of these ideas, including self contained proofs of
existence and uniqueness, we refer the reader to
\cite{aronsson:tourAMLE2004}.

There are, though, several variants to the classical Lipschitz
extension problem. A general formulation is the following: let
$(\X,d_{\X})$ and $(Z,d_Z)$ be two metric spaces, and define the
Lipschitz constant of a function $f: E \rightarrow Z$, $E \subset \X$,
as:
\begin{equation*}
\Lip(f;E) \triangleq \sup_{\substack{x,y \in E \\ x \neq y}} \frac{d_Z(f(x),f(y))}{d_{\X}(x,y)}.
\end{equation*}
Given a fixed pair of metric spaces $(\X,d_{\X})$ and $(Z,d_Z)$, as well as an
arbitrary function $f: E \rightarrow Z$ with $\Lip(f;E) < \infty$, one
can ask if it is possible to extend $f$ to a function $F: \X
\rightarrow Z$ such that $\Lip(F;\X) = \Lip(f;E)$. This is known as
the isometric Lipschitz extension problem (or property, if it is known
for a pair of metric spaces). Generally speaking it does not hold,
although certain special cases beyond $\X = \R^d$ and $Z = \R$ do
exist. For example, one can take $(\X,d_{\X})$ to be an arbitrary metric
space and $Z = \R$ (simply adapt \eqref{eqn: max lip ext} and
\eqref{eqn: min lip ext}). Another, more powerful generalization comes from
the work of Kirszbraun \cite{kirszbraun:lipschitzTransformations1934} (and later,
independently by Valentine
\cite{valentine:vectorLipschitzExt1945}). In his paper in $1934$, he
proved that if $\X$ and $Z$ are arbitrary Hilbert spaces, then they
have the isometric Lipschitz extension property. Further examples
exist; a more thorough discussion of the isometric Lipschitz extension
property can be found in \cite{wells:embeddingExtensions1975}.

For pairs of metric spaces with the isometric Lipschitz extension
property, one can then try to generalize the notion of an AMLE. Given
that an AMLE should locally be the best possible such extension, the
appropriate generalization for arbitrary metric spaces is the
following. Let $E \subset \X$ be closed, and suppose we are given a
function $f: E \rightarrow Z$ and a
minimal Lipschitz extension $U: \X \rightarrow Z$ such that
$\Lip(U;\X) = \Lip(f;E)$. The function $U$ is an AMLE if for every
open subset $V \subset \X \setminus E$ and every Lipschitz
mapping $\widetilde{U}: \X \rightarrow Z$ that coincides with $U$ on
$\X \setminus V$, we have
\begin{equation} \label{eqn: amle abstract def}
\Lip(U;V) \leq \Lip(\widetilde{U};V).
\end{equation}
When $(\X,d_{\X})$ is path connected, \eqref{eqn: amle abstract def} is
equivalent to the Aronsson condition that 
\begin{equation} \label{eqn: aronsson amle def}
\Lip(U;V) = \Lip(U;\partial V), \quad \text{for all open } V \subset
\X \setminus E.
\end{equation}
When $(\X,d_{\X})$ is an
arbitrary length space and $Z = \R$, there are several proofs of
existence of AMLE's
\cite{milman:absMinExtMetric1999,juutinen:amleMetricSpace2002,legruyer:amlePDE2007}
(some under certain conditions). The proof of uniqueness in this
scenario is given in \cite{peres:tugOfWar2009}.

Extending results on AMLE's to non scalar valued functions presents
many difficulties, which in turn has limited the number of results
along this avenue. Two recent papers have made significant progress,
though. In \cite{naor:treeAMLE2012}, the authors consider the case
when $(\X,d_{\X})$ is a locally compact length space, and $(Z,d_Z)$ is a
metric tree; they are able to prove existence and uniqueness of AMLE's
for this pairing. The case of vector valued functions with $(\X,d_{\X}) =
\R^d$ and $(Z,d_Z) = \R^m$ is considered in
\cite{sheffield:vectorAMLE2012}. In this case an AMLE is not
necessarily unique, so the authors propose a stronger condition called
tightness for which they are able to get existence and uniqueness
results in some cases. 

In this paper we seek to add to the progress on the theory of non
scalar valued AMLE's. We propose a generalized notion of an AMLE for a large
class of isometric Lipschitz extension problems, and prove a general theorem
for the existence of what we call a quasi-AMLE. A quasi-AMLE is,
essentially, a minimal Lipschitz extension that comes within
$\varepsilon$ of satisfying \eqref{eqn: amle abstract def}. We work not only with
general metric spaces, but also a general functional $\Phi$ that
replaces the specific functional $\Lip$. In our setting $\Lip$ is an
example of the type of functionals we consider, but others exist as
well.

One such example is given in
\cite{legruyer:minimalLipschitz1Field2009}. If we consider the classic
Lipschitz extension problem as the zero-order problem, then for the
first order problem we would want an extension that minimizes
$\Lip(\nabla F; \R^d)$. In this case, one is given
a subset $E \subset \R^d$ and a $1$-field $\PP_E = \{P_x\}_{x \in E}
\subset \PP^1(\R^d,\R)$, consisting of first order polynomials mapping
$\R^d$ to $\R$ that are indexed by the elements of $E$. The
goal is to extend $\PP_E$ to a function $F \in C^{1,1}(\R^d)$ such that
two conditions are satisfied: $1.)$ for each $x \in E$, the first order Taylor
polynomial $J_xF$ of $F$ at $x$ agrees with $P_x$; and $2.)$
$\Lip(\nabla F; \R^d)$ is minimal. By a result of Le Gruyer
\cite{legruyer:minimalLipschitz1Field2009}, such an extension is
guaranteed to exist with Lipschitz constant $\Gamma^1(\PP_E)$, assuming
that $\Gamma^1(\PP_E) < \infty$ (here $\Gamma^1$ is a functional defined
in \cite{legruyer:minimalLipschitz1Field2009}). The functional
$\Gamma^1$ can be thought of as the Lipschitz constant for
$1$-fields. By the results of this paper, one is guaranteed the
existence of a quasi-AMLE for this setting as well.

\section{Setup and the main theorem}

\subsection{Metric spaces} \label{sec: metric spaces}

Let $(\X,d_{\X})$ and $(Z,d_Z)$ be metric spaces. We will consider
functions of the form $f: E \rightarrow Z$, where $E \subset
\X$. For the range, we require:
\begin{enumerate}
\item
$(Z,d_Z)$ is a complete metric space.
\end{enumerate}
For the domain, $(\X,d_{\X})$, we require some additional geometrical properties:
\begin{enumerate}
\item
$(\X,d_{\X})$ is complete and proper (i.e., closed balls are compact).
\item
$(\X,d_{\X})$ is midpoint convex. Recall that this means that for any two
points $x,y \in \X$, $x \neq y$, there exists a third point $m(x,y) \in
\X$ for which 
\begin{equation*}
d_{\X}(x,m(x,y)) = d_{\X}(m(x,y),y) = \frac{1}{2} d_{\X}(x,y).
\end{equation*}
Such a point $m(x,y)$ is called the midpoint and $m: \X \times \X \rightarrow
\X$ is called the midpoint map. Since we have also assumed that
$(\X,d_{\X})$ is complete, this implies that $(\X,d_{\X})$ is a
geodesic (or strongly intrinsic) metric space. By definition then,
every two points $x,y \in \X$ are joined by a geodesic curve with
finite length equal to $d_{\X}(x,y)$. 
\item
$(\X,d_{\X})$ is distance convex, so that for all $x,y,z \in \X$, $x \neq y$,
\begin{equation*}
d_{\X}(m(x,y),z) \leq \frac{1}{2} (d_{\X}(x,z) + d_{\X}(y,z)).
\end{equation*}
Note that this implies that $(\X,d_{\X})$ is ball convex, which in turn
implies that every ball in $(\X,d_{\X})$ is totally convex. By definition,
this means that for any two points $x,y$ lying in a ball $B \subset \X$, the
geodesic connecting them lies entirely in $B$. Ball convexity also
implies that the midpoint map is unique, and, furthermore, since
$(\X,d_{\X})$ is also complete, that the geodesic between two points
is unique.
\end{enumerate}
We remark that the $(\X,d_{\X})$ is path connected by these assumptions,
and so \eqref{eqn: amle abstract def} is equivalent to \eqref{eqn:
  aronsson amle def} for all of the cases that we consider here.

\subsection{Notation}

Set $\N \triangleq \{0, 1, 2, \ldots \}$, $\N^* \triangleq \{1, 2,
3, \ldots \}$, and $\R^+ \triangleq [0,\infty)$. Let $S$ be an
arbitrary subset of $\X$, i.e., $S \subset \X$, and let $\mathring{S}$
and $\overline{S}$ denote the interior of $S$ and the closure of $S$,
respectively. For any $x \in \X$ and $S \subset \X$, set
\begin{equation*}
d_{\X}(x,S) \triangleq \inf \{ d_{\X}(x,y) \mid y \in S \}.
\end{equation*}
For each $x \in \X$ and $r > 0$, let $B(x;r)$ denote the open ball of
radius $r$ centered at $x$:
\begin{equation*}
B(x;r) \triangleq \{y \in \X \mid d_{\X}(x,y) < r\}.
\end{equation*}
We will often utilize a particular type of ball: for any $x,y \in \X$, define 
\begin{equation*}
B_{1/2}(x,y) \triangleq B\left(m(x,y); \frac{1}{2}d_{\X}(x,y) \right).
\end{equation*}
By $\F(\X,Z)$, we denote the space of functions mapping subsets
of $\X$ into $Z$:
\begin{equation*}
\F(\X,Z) \triangleq \{f:E \rightarrow Z \mid E \subset \X\}.
\end{equation*}
If $f \in \F(\X,Z)$, set $\dom(f)$ to be the domain of $f$. We 
use $E = \dom(f)$ interchangeably depending on the situation. We also
set $\K(\X)$ to be the set of all compact subsets of $\X$. 

\subsection{General Lipschitz extensions}

\begin{definition}
Given $f \in \F(\X,Z)$, a function $F \in \F(\X,Z)$ is an {\it
  extension} of $f$ if
\begin{equation*}
\dom(f) \subset \dom(F) \quad \text{and} \quad F(x) = f(x), \text{ for
  all } x \in \dom(f).
\end{equation*}
\end{definition}

We shall be interested in arbitrary functionals $\Phi$ with domain $\F(\X,Z)$ such that:
\begin{align*}
\Phi: \F(\X,Z) &\rightarrow \F(\X \times \X, \R^+ \cup \{\infty\}) \\
f &\mapsto \Phi(f): \dom(f) \times \dom(f) \rightarrow
\R^+ \cup \{ \infty \}.
\end{align*}
In order to simplify the notation slightly, for any $f \in \F(\X,Z)$
and $x,y \in \dom(f)$, we set
\begin{equation*}
\Phi(f;x,y) \triangleq \Phi(f)(x,y).
\end{equation*}
We also extend the map $\Phi(f)$ to subsets $D \subset \dom(f)$ as follows:
\begin{equation*}
\Phi(f;D) \triangleq \sup_{\substack{x,y \in D \\ x \neq y}} \Phi(f;x,y).
\end{equation*}
The map $\Phi$ serves as a generalization of the standard Lipschitz constant
$\Lip(f;D)$ first introduced in Section \ref{sec: introduction}. As such, one can
think of it in the context of minimal extensions. Let $\F_{\Phi}(\X,Z)
\subset \F(\X,Z)$ denote those functions in $\F(\X,Z)$ for which $\Phi$
is finite, i.e.,
\begin{equation*}
\F_{\Phi}(\X,Z) \triangleq \{ f \in \F(\X,Z) \mid \Phi(f; \dom(f)) < \infty \}.
\end{equation*}
We then have the following definition.

\begin{definition} \label{defn: minimal extension}
Let $f \in \F_{\Phi}(\X,Z)$ and let $F \in \F_{\Phi}(\X,Z)$ be an extension of
$f$. We say $F$ is a {\it minimal extension} of the function $f$ if
\begin{equation} \label{eqn: min ext property}
\Phi(F;\dom(F)) = \Phi(f; \dom(f)).
\end{equation}
\end{definition}

One can then generalize the notion of an absolutely minimal
Lipschitz extension (AMLE) in the following way:
\begin{definition} \label{def: abstract general amle def}
Let $f \in \F_{\Phi}(\X,Z)$ with $\dom(f)$ closed and let $U \in \F_{\Phi}(\X,Z)$ be a
minimal extension of $f$ with $\dom(U) = \X$. Then $U$ is an {\it
  absolutely minimal Lipschitz extension} of $f$ if for every
open set $V \subset \X \setminus \dom(f)$ and every $\widetilde{U} \in
\F_{\Phi}(\X,Z)$ with $\dom(\widetilde{U}) = \X$ that coincides with
$U$ on $\X \setminus V$, 
\begin{equation*}
\Phi(U;V) \leq \Phi(\widetilde{U};V).
\end{equation*}
\end{definition}
Alternatively, we can extend Aronsson's original definition of AMLEs:
\begin{definition} \label{def: aronsson general amle def}
Let $f \in \F_{\Phi}(\X,Z)$ with $\dom(f)$ closed and let $U \in
\F_{\Phi}(\X,Z)$ be a minimal extension
of $f$ with $\dom(U) = \X$. Then $U$ is an {\it absolutely minimal Lipschitz
  extension} of $f$ if
\begin{equation} \label{eqn: general amle condition}
\Phi(U; V) = \Phi(U; \partial V), \quad \text{for all open }
V \subset \X \setminus \dom(f).
\end{equation}
\end{definition}
In fact, since we have assumed that $(\X,d_{\X})$ is path connected,
Definitions \ref{def: abstract general amle def} and \ref{def:
  aronsson general amle def} are equivalent; see Appendix \ref{sec:
  equiv amle defs} for the details.

In this paper we prove the existence of a function
$U$ that is a minimal extension of $f$, and that ``nearly''
satisfies \eqref{eqn: general amle condition}. In order to make this
statement precise, we first specify the properties that $\Phi$ must
satisfy, and then formalize what we mean by ``nearly.'' Before we get
to either task, though, we first define the following family of curves.

\begin{definition}\label{def: gamma curve}
For each $x,y \in \X$, $x \neq y$, let $\Gamma(x,y)$ denote
the set of curves
\begin{equation*}
\gamma: [0,1] \rightarrow \overline{B}_{1/2}(x,y),
\end{equation*}
such that $\gamma(0) = x$, $\gamma(1) = y$, $\gamma$ is continuous,
and $\gamma$ is monotone in the following sense:
\begin{equation*}
\text{If } 0 \leq t_1 < t_2 \leq 1, \text{ then } d_{\X}(\gamma(0),
\gamma(t_1)) < d_{\X}(\gamma(0), \gamma(t_2)).
\end{equation*}
\end{definition}

The required properties of $\Phi$ are the following (note that \ref{P1} has
already been stated as a definition):

\begin{enumerate}
\setcounter{enumi}{-1}
\renewcommand{\theenumi}{$(P_{\arabic{enumi}})$}
\renewcommand{\labelenumi}{\theenumi}
\item\label{P0}
$\Phi$ is symmetric and nonnegative: \\
For all $f \in \F(\X,Z)$ and for all $x,y \in \dom(f)$,
\begin{equation*}
\Phi(f;x,y) =  \Phi(f;y,x) \geq 0.
\end{equation*}

\item\label{P1}
Pointwise evaluation: \\
For all $f \in \F(\X,Z)$ and for all $D \subset \dom(f)$,
\begin{equation*}
\Phi(f;D) = \sup_{\substack{x,y \in D \\ x \neq y}} \Phi(f;x,y).
\end{equation*}

\item\label{P2}
$\Phi$ is minimal: \\
For all $f \in \F_{\Phi}(\X,Z)$ and for all $D \subset \X$ such that $\dom(f)
\subset D$, there exists an extension $F: D \rightarrow Z$ of $f$
such that
\begin{equation*}
\Phi(F;D) = \Phi(f;\dom(f)).
\end{equation*}

\item\label{P3}
Chasles' inequality: \\
For all $f \in \F_{\Phi}(\X,Z)$ and for all $x,y \in \dom(f)$, $x \neq y$, such that
$\overline{B}_{1/2}(x,y) \subset \dom(f)$, there exists a curve $\gamma \in
\Gamma (x,y)$ such that
\begin{equation*}
\Phi(f;x,y) \leq \inf_{t \in [0,1]} \max \left\{
\Phi(f;x,\gamma(t)), \Phi(f;\gamma(t),y) \right\}.
\end{equation*}

\item\label{P4}
Continuity of $\Phi$: \\
Let $f \in \F_{\Phi}(\X,Z)$. For each $x,y \in \dom(f)$,
$x \neq y$, and for all $\varepsilon > 0$, there exists
$\eta = \eta(\varepsilon, d_{\X}(x,y)) > 0$ such that 
\begin{equation*}
\text{For all } z \in B(y;\eta) \cap \dom(f), \quad |\Phi(f;x,y) -
\Phi(f;x,z)| < \varepsilon.
\end{equation*}

\item\label{P5}
Continuity of $f$: \\
If $f \in \F_{\Phi}(\X,Z)$, then $f: \dom(f) \rightarrow Z$ is a
continuous function.
\end{enumerate}

\begin{remark}
Property \ref{P3} is named after French mathematician Michel Chasles.
\end{remark}

\subsection{Examples of the metric spaces and the functional $\Phi$}

Before moving on, we give some examples of the metric spaces $(\X,d_{\X})$
and $(Z,d_Z)$ along with the functional $\Phi$.

\subsubsection{Scalar valued Lipschitz extensions}

The scalar valued case discussed at the outset is one example. Indeed,
one can take $\X = \R^d$ and $d_{\X}(x,y) = \|x-y\|$, where $\| \cdot
\|$ is the Euclidean distance. For the range, set $Z = \R$ and $d_Z(a,b) =
|a-b|$. For any $f:E \rightarrow \R$, where $E \subset \X$, define
$\Phi$ as:
\begin{align*}
\Phi(f;x,y) &= \Lip(f;x,y) \triangleq \frac{|f(x) - f(y)|}{\|x-y\|}, \\
\Phi(f;E) &= \Lip(f;E) \triangleq \sup_{\substack{x,y \in E \\ x \neq y}} \Lip(f;x,y).
\end{align*}

Clearly \ref{P0} and \ref{P1} are satisfied. By the work of McShane
\cite{mcshane:extensionRangeFcns1934} and Whitney
\cite{whitney:analyticExtensions1934}, \ref{P2} is also
satisfied. \ref{P3} is satisfied with $\gamma(t) = (1-t)x + ty$, and
\ref{P4} and \ref{P5} are easy to verify.

\subsubsection{Lipschitz mappings between Hilbert spaces}

More generally, one can take $(\X,d_{\X}) = \HH_1$ and $(Z,d_Z) = \HH_2$,
where $\HH_1$ and $\HH_2$ are Hilbert spaces (note, since we assume
that $(\X,d_{\X})$ is proper, there are some restrictions on $\HH_1$). Then
for any $f: E \rightarrow \HH_2$, with $E \subset \HH_1$, define $\Phi$ as:
\begin{align*}
\Phi(f;x,y) &= \Lip(f;x,y) \triangleq
\frac{\|f(x)-f(y)\|_{\HH_2}}{\|x-y\|_{\HH_1}}, \\
\Phi(f;E) &= \Lip(f;E) \triangleq \sup_{\substack{x,y \in E \\ x \neq
    y}} \Lip(f;x,y).
\end{align*}

Clearly \ref{P0} and \ref{P1} are satisfied. By the work of Kirszbraun
\cite{kirszbraun:lipschitzTransformations1934} and later Valentine
\cite{valentine:vectorLipschitzExt1945}, \ref{P2} is also
satisfied. \ref{P3} is satisfied with $\gamma(t) = (1-t)x + ty$, and
\ref{P4} and \ref{P5} are easy to verify.

\subsubsection{Lipschitz mappings between metric spaces}

More generally still, one can take any pair of metric spaces $(\X,d_{\X})$
and $(Z,d_Z)$ satisfying the assumptions of Section \ref{sec: metric
  spaces}. Clearly \ref{P0}, \ref{P1}, \ref{P4}, and \ref{P5} are satisfied. For
\ref{P3}, we can take $\gamma \in \Gamma(x,y)$ to be the unique
geodesic between $x$ and $y$. All that remains to check, then, is
\ref{P2}, the existence of a minimal extension. Such a condition is not
satisfied between two metric spaces in general, although examples
beyond those already mentioned do exist. For example, one can take
$(\X,d_{\X})$ to be any metric space and $(Z,d_Z) = \ell_n^{\infty}$,
where $\ell_n^{\infty}$ denotes $\R^n$ with the norm $\|x\|_{\infty}
\triangleq \max\{|x_j| \mid j = 1, \ldots, n\}$. See
\cite{wells:embeddingExtensions1975}, Theorem $11.2$, Chapter $3$, as
well as the discussion afterwards . See also
\cite{nachbin:hahnBanachLinearTrans1950,kelley:banachExtension1952}.

\subsubsection{$1$-fields} \label{sec: 1-fields}

Let $(\X,d_{\X}) = \R^d$ endowed with the Euclidean metric. Set
$\PP^1(\R^d,\R)$ to be the set of first degree polynomials (affine
functions) mapping $\R^d$ to $\R$. We take $Z = \PP^1(\R^d,\R)$, and
write each $P \in \PP^1(\R^d,\R)$ in the following form:
\begin{equation*}
P(a) = p_0 + D_0p \cdot a, \quad p_0 \in \R, \enspace D_0p \in \R^d,
\enspace a \in \R^d.
\end{equation*}
For any $P,Q \in \PP^1(\R^d,\R)$, we then define $d_Z$ as:
\begin{equation*}
d_Z(P,Q) \triangleq |p_0 - q_0| + \|D_0p - D_0q\|,
\end{equation*}
where $|\cdot|$ is just the absolute value, and $\|\cdot\|$ is the
Euclidean distance on $\R^d$.

For a function $f \in \F(\X,Z)$, we use the following notation (note,
as usual, $E \subset \X$):
\begin{align*}
f: \, &E \rightarrow \PP^1(\R^d,\R) \\
&x \mapsto f(x)(a) = f_x + D_xf \cdot (a - x),
\end{align*}
where $f_x \in \R$, $D_xf \in \R^d$, and $a \in \R^d$ is the
evaluation variable of the polynomial $f(x)$. Define the functional
$\Phi$ as:
\begin{align*}
\Phi(f;x,y) &= \Gamma^1(f;x,y) \triangleq 2 \sup_{a \in \R^d} \frac{|f(x)(a) -
  f(y)(a)|}{\|x-a\|^2 + \|y-a\|^2}, \\
\Phi(f;E) &= \Gamma^1(f;E) \triangleq \sup_{\substack{x,y \in E \\ x
    \neq y}} \Gamma^1(f;x,y). 
\end{align*}

Using the results contained in
\cite{legruyer:minimalLipschitz1Field2009}, one can show that for
these two metric spaces and for this definition of $\Phi$, that properties
\ref{P0}-\ref{P5} are satisfied; the full details are given in Appendix \ref{sec:
  proof of P3 for 1 fields}. In particular, there exists an extension
$U: \R^d \rightarrow \PP^1(\R^d,\R)$, $U(x)(a) = U_x + D_xU \cdot (a-x)$,
such that $U(x) = f(x)$ for all $x \in E$ and $\Phi(U;\R^d) =
\Phi(f;E)$. Furthermore, define the function $F: \R^d \rightarrow \R$ as
$F(x) = U_x$ for all $x \in \R^d$. Note that $F \in C^{1,1}(\R^d)$, and
set for each $x \in \R^d$, $J_xF(a) \triangleq F(x) + \nabla F(x) \cdot (a-x) \in
\PP^1(\R^d,\R)$ to be the first order Taylor expansion of $F$ around
$x$. Then $F$ satisfies the following properties:
\begin{enumerate}
\item
$J_xF = f(x)$ for all $x \in E$.
\item
$\Lip(\nabla F) = \Gamma^1(f;E)$.
\item
If $\widetilde{F} \in C^{1,1}(\R^d)$ also satisfies $J_x\widetilde{F}
= f(x)$ for all $x \in E$, then $\Lip(\nabla F) \leq \Lip(\nabla
\widetilde{F})$. 
\end{enumerate}
Thus $F$ is the extension of the $1$-field $f$ with minimum Lipschitz
derivative (see \cite{legruyer:minimalLipschitz1Field2009} for the
proofs and a complete explanation). The $1$-field $U$ is the corresponding set of jets of
$F$. For an explicit construction of $F$ when $E$ is finite we refer
the reader to \cite{wells:diffFcnsLipDer1973}. 

\subsubsection{$m$-fields}

A similar result for $m$-fields, where $m \geq 2$, is an open
problem. In particular, it is unknown what the correct corresponding
functional $\Phi = \Gamma^m$ is. It seems plausible, though, that such
a functional will satisfy the properties \ref{P0}-\ref{P5}.

\subsection{Main theorem}

The AMLE condition \eqref{eqn: general amle condition} is for any open
set off of the domain of the initial function $f$. In our analysis, we
look at subfamily of open sets that approximates the family of all open
sets. In particular, we look at finite unions of open balls. The
number of balls in a particular union is capped by a universal
constant, and furthermore, the radius of each ball must also be larger than
some constant. For any $\rho > 0$ and $N_0 \in \N$, define such a
collection as:
\begin{equation*}
\OO(\rho,N_0) \triangleq \left\{ \Omega = \bigcup_{n=1}^N B(x_n;r_n)
  \mid x_n \in \X, \enspace r_n \geq \rho, \enspace N \leq N_0 \right\}.
\end{equation*}
Note that as $\rho \rightarrow 0$ and $N_0 \rightarrow \infty$,
$\OO(\rho,N_0)$ contains all open sets if $(\X,d_{\X})$ is compact. We shall always use $\Omega$
to denote sets taken from $\OO(\rho,N_0)$. For any such set, we use
$\RR(\Omega)$ to denote the collection of balls that make up $\Omega$:
\begin{equation*}
\RR(\Omega) = \left\{ B(x_n;r_n) \mid n = 1, \ldots, N, \enspace \Omega =
\bigcup_{n=1}^N B(x_n;r_n) \right\}.
\end{equation*}

We also define, for any $f \in \F(\X,Z)$, any open $V \subset
\dom(f)$, $V \neq \X$, and any $\alpha > 0$, the following
approximation of $\Phi(f;V)$:
\begin{equation*}
\Psi(f;V;\alpha) \triangleq \sup \left\{ \Phi(f;x,y) \mid B(x;d_{\X}(x,y))
  \subset V, \enspace d_{\X}(x,\partial V) \geq \alpha \right\}.
\end{equation*}
Using these two approximations, our primary result is the following:

\begin{theorem}\label{Main.theo1}
Let $(\X,d_{\X})$ and $(Z,d_Z)$ be metric spaces satisfying the assumptions
of Section \ref{sec: metric spaces}, let $\Phi$ be a functional
satisfying properties \ref{P0}-\ref{P5}, and let $X \in \K(\X)$. Given $f \in \F_{\Phi}(X,Z)$,
$\rho > 0$, $N_0 \in \N$, $\alpha > 0$, and $\sigma_0 > 0$, there
exists $U = U(f, \rho, N_0, \alpha, \sigma_0) \in \F_{\Phi}(X,Z)$ such that
\begin{enumerate}
\item
$U$ is a minimal extension of $f$ to $X$; that is,
\begin{align*}
&\dom(U) = X, \\
&U(x) = f(x), \text{ for all } x \in \dom(f), \\
&\Phi(U;X) = \Phi(f;\dom(f)).
\end{align*}
\item
The following quasi-AMLE condition is satisfied on $X$:
\begin{equation} \label{eqn: quasi-amle condition}
\Psi(U;\Omega;\alpha) - \Phi(U;\partial \Omega) < \sigma_0, \quad
\text{for all } \Omega \in \OO(\rho,N_0), \enspace \Omega \subset X
\setminus \dom(f).
\end{equation}
\end{enumerate}
\end{theorem}

We call such extensions quasi-AMLEs, and view them as a first step
toward proving the existence of AMLEs under these general
conditions. We note that there are essentially four areas of
approximation. The first is that we extend to an arbitrary, but fixed
compact set $X \subset \X$ as opposed to the entire space. The second
was already mentioned; rather than look at
all open sets, we look at those belonging to $\OO(\rho,N_0)$. Since
$X$ is compact, as $\rho \rightarrow 0$ and $N_0 \rightarrow
\infty$, $\OO(\rho,N_0)$ will contain all open sets in $X$. Third, we
allow ourselves a certain amount of error with the parameter
$\sigma_0$. As $\sigma_0 \rightarrow 0$, the values of the Lipschitz
constants on $\Omega$ and $\partial \Omega$ should coincide. The last
part of the approximation is the use of the functional $\Psi$ to
approximate $\Phi$ on each $\Omega \in \OO(\rho,N_0)$. While this may
at first not seem as natural as the other areas of approximation,
the following proposition shows that in fact $\Psi$ works rather well
in the context of the AMLE problem.

\begin{proposition} \label{prop: alpha to zero}
Let $f \in \F_{\Phi}(\X,Z)$. For any open $V \subset \dom(f)$,
$V \neq \X$, and $\alpha \geq 0$, let us define
\begin{equation*}
V_{\alpha}  \triangleq \{ x \in V \mid d_{\X}(x,\partial V) \geq \alpha \}.
\end{equation*}
Then for all $ \alpha > 0$, 
\begin{equation}\label{Psi1}
\Phi(f;V_{\alpha}) \leq \max \{\Psi(f;V;\alpha), \Phi(f;\partial V_{\alpha})\},
\end{equation}
and
\begin{equation}\label{Psi2}
\Phi(f;V) = \max \{\Psi(f;V;0), \Phi(f;\partial V)\}.
\end{equation}
\end{proposition}

\begin{proof}
See Appendix \ref{sec: proof of alpha to zero}.
\end{proof}

Proposition \ref{prop: alpha to zero}, along with the discussion
immediately preceding it, seems to indicate that if one were able to
pass through the various
limits to obtain $U(f,\rho,N_0,\alpha,\sigma_0) \rightarrow U(f)$ as $\rho
\rightarrow 0$, $N_0 \rightarrow \infty$, $\alpha \rightarrow 0$, and
$\sigma_0 \rightarrow 0$, then one would have a general theorem of
existence of AMLEs for suitable pairs of metric spaces and Lipschitz-type
functionals. Whether such a procedure is in fact possible, though, is
yet to be determined.
 
The proof of Theorem \ref{Main.theo1} is given in
Section \ref{sec: proof of sigma0 quasi AMLE}, with the relevant
lemmas stated and proved in Section \ref{sec: lemmas for sigma0 AMLE
  theorem}. The main ideas of the proof are as follows. Using \ref{P2},
we can find a minimal extension $U_0 \in \F_{\Phi}(X,Z)$ of
$f$ with $\dom(U_0) = X$. If such an extension also satisfies \eqref{eqn: quasi-amle
  condition}, then we take $U = U_0$ and we are finished. If, on
the other hand, $U_0$ does not satisfy \eqref{eqn: quasi-amle
  condition}, then there must be some $\Omega_1 \in \OO(\rho,N_0)$,
$\Omega_1 \subset X \setminus \dom(f)$, for which
$\Psi(U_0;\Omega_1;\alpha) - \Phi(U_0;\partial\Omega_1) \geq
\sigma_0$. We derive a new minimal extension $U_1 \in
\F_{\Phi}(X,Z)$ of $f$ from $U_0$ by correcting $U_0$ on
$\Omega_1$. To perform the correction, we restrict ourselves to
$U_0|_{\partial \Omega_1}$, and extend this function to $\Omega_1$
using once again \ref{P2}. We then patch this extension into $U_0$,
giving us $U_1$. We then ask if $U_1$ satisfies \eqref{eqn: quasi-amle
  condition}. If it does, we take $U = U_1$ and we are
finished. If it does not, we repeat the procedure just outlined. The
main work of the proof goes into showing that the repetition of
such a procedure must end after a finite number of iterations. 

It is also interesting to note that the extension procedure itself is
a ``black box.'' We do not have any knowledge of the behavior of the
extension outside of \ref{P0}-\ref{P5}, only that it exists. We then refine this extension by using
local extensions to correct in areas that do not satisfy the
quasi-AMLE condition. The proof then is not about the extension of
functions, but rather the refinement of such extensions.

\begin{remark}
The procedure outlined above seems to indicate that the proof can be
adapted numerically for applications in which one needs to compute a generalized
AMLE to within some tolerance. Indeed, if one is able to numerically
compute the corrections efficiently, then by the proof, one is
guaranteed to have an algorithm with finite stopping time for any set
of prescribed tolerances $\rho$, $N_0$, $\alpha$, and $\sigma_0$.
\end{remark}

\section{Proof of Theorem \ref{Main.theo1}: Existence of
  quasi-AMLE's} \label{sec: proof of sigma0 quasi AMLE}

In this section we outline the key parts of the proof of Theorem
\ref{Main.theo1}. We begin by defining a local correction operator that we will use
repeatedly. 

\subsection{Definition of the correction operator $H$} \label{sec: def of H}

\begin{definition}\label{Def.H}
Let $X \in \K(\X)$, $f \in \F_{\Phi}(X,Z)$, and $\Omega \in \OO(\rho,N_0)$ with
$\overline{\Omega} \subset \dom(f)$. By \ref{P2} there exists an $F \in
\F_{\Phi}(X,Z)$ with $\dom(F) = \overline{\Omega}$ such that 
\begin{align}
&F(x) = f(x), \text{ for all } x \in \partial \Omega, \nonumber \\
&\Phi(F;\Omega) = \Phi(f;\partial \Omega). \label{H.2}
\end{align}
Given such an $f$ and $\Omega$, define the operator $H$ as:
\begin{equation}\label{H.1}
 H(f;\Omega)(x) \triangleq F(x), \quad \text{for all } x \in \Omega.
\end{equation}
\end{definition}

\subsection{A sequence of total, minimal extensions} \label{sec: sequence of
  min extensions}

Fix the metric spaces $(\X,d_{\X})$ and $(Z,d_Z)$, the Lipschitz functional
$\Phi$, the compact domain $X \in \K(\X)$, as well as $f \in
\F_{\Phi}(X,Z)$, $\rho > 0$, $N_0 \in \N$, $\alpha > 0$, $\sigma_0 >
0$. Set:
\begin{equation*}
K \triangleq \Phi(f;\dom(f)).
\end{equation*}

Using \ref{P2}, let $U_0 \in \F_{\Phi}(X,Z)$ be a minimal
extension of $f$ to all of $X$; recall that this means:
\begin{align*}
&\dom(U_0) = X, \\
&U_0(x) = f(x), \text{ for all } x \in \dom(f), \\
&\Phi(U_0;X) = \Phi(f;\dom(f)).
\end{align*}
We are going to recursively construct a sequence $\{U_n\}_{n \in \N}$
of minimal extensions of $f$ to $X$. First, for any $n \in \N$, define
\begin{equation*}
\Delta_n \triangleq \{ \Omega \in \OO(\rho,N_0) \mid
\Psi(U_n;\Omega;\alpha) - \Phi(U_n;\partial \Omega) \geq \sigma_0,
\enspace \Omega \subset X \setminus \dom(f) \}.
\end{equation*}
The set $\Delta_n$ contains all admissible open sets for
which the extension $U_n$ violates the quasi-AMLE condition. If $\Delta_n =
\emptyset$, then we can take $U = U_n$ and we are finished.

If, on the other hand, $\Delta_n \neq \emptyset$, then to obtain
$U_{n+1}$ we take $U_n$ and pick any $\Omega_{n+1} \in \Delta_n$ and
set
\begin{equation*}
U_{n+1}(x) \triangleq \left\{
\begin{array}{ll}
H(U_n;\Omega_{n+1})(x), & \text{if } x \in \Omega_{n+1}, \\
U_n(x) & \text{if } x \in X \setminus \Omega_{n+1},
\end{array}
\right.
\end{equation*}
where $H$ was defined in Section \ref{sec: def of H}. Thus, along with
$\{ U_n \}_{n \in \N}$, we also have a sequence of refining sets $\{
\Omega_n \}_{n \in \N^*}$ such that $\Omega_n \in \OO(\rho,N_0)$,
$\Omega_n \subset X \setminus \dom(f)$, and $\Omega_n \in
\Delta_{n-1}$ for all $n \in \N^*$.

Since $\dom(U_0) = X$, and since $\Omega_n \subset X
\setminus \dom(f)$, we see by construction
that $\dom(U_n) = X$ for all $n \in \N$. By the arguments in
Section \ref{sec: H preserves Lip} and Lemma \ref{douanes_lem_0}
contained within, we see that each of the functions $U_n$ is also a
minimal extension of $f$.

\subsection{Reducing the Lipschitz constant on the refining sets
  $\{\Omega_n\}_{n \in \N^*}$}

Since each of the functions $U_n$ is a minimal extension of $f
\in \F_{\Phi}(X,Z)$, we have
\begin{equation} \label{eqn: lip const of Un}
\Phi(U_n;X) = K, \quad \text{for all } n \in \N.
\end{equation}
Furthermore, since $\Omega_{n+1} \in \Delta_n$, we have by definition,
\begin{equation} \label{eqn: Delta def rewrite}
\Psi(U_n;\Omega_{n+1};\alpha) - \Phi(U_n;\partial \Omega_{n+1}) \geq \sigma_0.
\end{equation}
Using the definition of the operator $H$ and \eqref{eqn: Delta def
  rewrite}, we also have for any $n \in \N^*$,
\begin{equation} \label{Main_3}
\Phi(U_n;\Omega_n) = \Phi(H(U_{n-1};\Omega_n);\Omega_n) =
\Phi(U_{n-1};\partial \Omega_n) \leq \Psi(U_{n-1};\Omega_n;\alpha) - \sigma_0.
\end{equation}
Furthermore, combining \eqref{eqn: lip const of Un} and
\eqref{Main_3}, and using property \ref{P1} as well as the definition
of $\Psi$, one can arrive at the following:
\begin{equation} \label{Main_4}
\Phi(U_n; \Omega_n) \leq K - \sigma_0, \quad \text{for all } n \in \N^*.
\end{equation}
Thus we see that locally on $\Omega_n$, the total, minimal extension
$U_n$ is guaranteed to have Lipschitz constant bounded by
$K-\sigma_0$. In fact we can say much more.

\begin{lemma} \label{lem: Qp}
The following property holds true for all $p \in \N^*$:
\begin{equation*}\label{eqn: Qp}
\exists \, M_p \in \N^*: \forall \, n > M_p, \quad
\Phi(U_n;\Omega_n) < K - p \frac{\sigma_0}{2}.
\tag{$\Q_p$}\end{equation*}
\end{lemma}

The property \eqref{eqn: Qp} is enough to prove Theorem
\ref{Main.theo1}. Indeed, if $\Delta_n \neq \emptyset$ for all $n \in
\N$, then by \eqref{eqn: Qp} one will have $\Phi(U_n;\Omega_n) < 0$ for $n$ sufficiently
large. However, by the definition of $\Phi$ one must have
$\Phi(U_n;\Omega_n) \geq 0$, and so we have arrived at a
contradiction. Now for the proof of Lemma \ref{lem: Qp}.

\begin{proof}
We prove \eqref{eqn: Qp} by induction. By \eqref{Main_4}, it is clearly true
for $p=1$. Let $p \geq 2$ and suppose that $(\Q_{p-1})$ is true; we
wish to show that \eqref{eqn: Qp} is true as well. Let $M_{p-1}$ be an
integer satisfying $(\Q_{p-1})$ and assume that $\Delta_{M_{p-1}} \neq
\emptyset$. Let us define the following sets:
\begin{align*}
\A_{p,n} &\triangleq \bigcup \{\Omega_m \mid M_{p-1} < m \leq n\}, \\
\A_{p,\infty} &\triangleq \bigcup \{\Omega_m \mid M_{p-1} < m\}, \\
\widetilde{\RR}(\A_{p,\infty}) &\triangleq \{B(x;r) \mid \exists \,
m > M_{p-1} \text{ with } B(x;r) \in \RR(\Omega_m) \}.
\end{align*}

The closure of each set $\A_{p,n}$ is compact and the sequence $\{\A_{p,n}\}_{n >
  M_{p-1}}$ is monotonic under inclusion and converges to
$\A_{p,\infty}$ in Hausdorff distance as $n \rightarrow \infty$. In
particular, for $\varepsilon > 0$, there exists $N_p > M_{p-1}$ such that
\begin{equation*}
\delta (\A_{p,N_p}, \A_{p,\infty}) \leq \varepsilon,
\end{equation*}
where $\delta$ is the Hausdorff distance.

Now apply the Geometrical Lemma \ref{lem: geometrical lemma} (see
Section \ref{sec: geometrical lemma}) to the
sets $\A_{p,n}$ and $\A_{p,\infty}$ with $\beta = \alpha -
\varepsilon$. One obtains $N_{\varepsilon} \in \N$ such that
\begin{equation} \label{eqn: apply geo lemma}
\forall \, B(x;r), \text{ if } r \geq \alpha - \varepsilon \text{
  and } B(x;r) \subset \A_{p,\infty}, \text{ then } B(x;r-\varepsilon)
\subset \A_{p,N_{\varepsilon}}.
\end{equation}
Take $M_p \triangleq \max\{N_p,N_{\varepsilon}\}$. One can then obtain
the following lemma, which is essentially a
corollary of \eqref{eqn: apply geo lemma}.

\begin{lemma}\label{Boule1}
For all $n >M_p$ and for all $B(x;r) \subset \Omega_n$ with
$d_{\X}(x,\partial\Omega_n)\geq \alpha$ and $r < \alpha$, we have
\begin{equation*}
\text{if } B(x;r) \not\subset \A_{p,M_p}, \text{ then }  r \geq \alpha-\varepsilon.
\end{equation*}
\end{lemma}

\begin{proof}
 Let $B(x;r)$ be a ball that satisfies the hypothesis of the lemma and
 suppose $ B(x;r) \not\subset \A_{p,M_p}$. Since $d_{\X}(x, \partial \Omega_n)\geq
 \alpha$ and $B(x;r) \subset B(x;\alpha) \subset \Omega_n$, we have
 $B(x;\alpha) \not\subset \A_{p,M_p}$. On the other hand, $B(x;\alpha)
 \subset \A_{p,\infty}$ and (trivially) $\alpha \geq
 \alpha-\varepsilon$, so by \eqref{eqn: apply geo lemma},
 $B(x;\alpha-\varepsilon) \subset \A_{p,M_p}$. Therefore $r \geq
 \alpha-\varepsilon$. \qed
\end{proof}

Now let us use the inductive hypothesis $(\Q_{p-1})$. Let $n > M_p$
and let $x,y \in \Omega_n$ such that $B(x;d_{\X}(x,y)) \subset \Omega_n$
with $d_{\X}(x,\partial\Omega_n) \geq\alpha$.

\begin{case}
Suppose that $B(x;d_{\X}(x,y)) \subset \A_{p,M_p}$.
In this case we apply the Customs Lemma (Lemma \ref{lem: customs lemma}) with
$\A = \A_{p,n-1}$ (see Section \ref{sec: customs lemma}). Since $n > M_p > M_{p-1}$, we are assured by the inductive hypothesis that 
\begin{equation*}
\Phi(U_j;\Omega_j) < K - (p-1)\frac{\sigma_0}{2}, \quad \text{for all
} j = M_{p-1}+1, \ldots, n-1. 
\end{equation*}
Thus we can conclude from the Customs Lemma that
\begin{equation} \label{eqn: case 1 final upper bound}
\Phi(U_{n-1};x,y) \leq K - (p-1)\frac{\sigma_0}{2}.
\end{equation}
That completes the first case.
\end{case}

\begin{case}
Suppose that  $B(x;d_{\X}(x,y)) \not\subset \A_{p,M_p}$. By Lemma
\ref{Boule1}, we know that $d_{\X}(x,y) \geq \alpha-\varepsilon$. Thus, by
\eqref{eqn: apply geo lemma}, we have $B(x;d_{\X}(x,y)-2\varepsilon)
\subset \A_{p,M_p}$. 

Let $\gamma \in \Gamma(x,y)$ be the curve satisfying \ref{P3} and set
\begin{equation*}
y_1 \triangleq \partial B(x;d_{\X}(x,y)-2\varepsilon) \cap \gamma.
\end{equation*}
Write $\Phi(U_{n-1};x,y)$ in the form:
\begin{equation*}
\Phi(U_{n-1};x,y) = \Phi(U_{n-1};x,y) - \Phi(U_{n-1};x,y_1) + \Phi(U_{n-1};x,y_1).
\end{equation*}
Using \ref{P4} and the fact that $d_{\X}(x,y) \geq \alpha-\varepsilon$,
there exists a constant $C(\varepsilon,\alpha)$ satisfying
$C(\varepsilon,\alpha) \rightarrow 0$ as $\varepsilon \rightarrow 0$
such that
\begin{equation} \label{eqn: bound on small piece}
\Phi(U_{n-1};x,y) - \Phi(U_{n-1};x,y_1) \leq C(\varepsilon,\alpha).
\end{equation}
Moreover, since $B(x;d_{\X}(x,y_1)) \subset \A_{p,M_p}$, we can apply the
Customs Lemma along with the inductive
hypothesis $(\Q_{p-1})$ (as in the first case) to conclude that
\begin{equation} \label{eqn: bound on slightly smaller half ball}
\Phi(U_{n-1};x,y_1) \leq K - (p-1)\frac{\sigma_0}{2}.
\end{equation}
Combining \eqref{eqn: bound on small piece} and \eqref{eqn: bound on
  slightly smaller half ball} we obtain 
\begin{equation*}
\Phi(U_{n-1};x,y) \leq K - (p-1)\frac{\sigma_0}{2} + C(\varepsilon,\alpha).
\end{equation*}
Since we can choose $\varepsilon$ such that $C(\varepsilon,\alpha)
\leq \sigma_0/2$, we have
\begin{equation} \label{eqn: case 2 final upper bound}
\Phi(U_{n-1};x,y) \leq K - p\frac{\sigma_0}{2} + \sigma_0.
\end{equation}
That completes the second case.
\end{case}

Now using \eqref{eqn: case 1 final upper bound} in the first case and
\eqref{eqn: case 2 final upper bound} in the second case we obtain 
\begin{equation} \label{eqn: both cases upper bound}
\Psi(U_{n-1};\Omega_n;\alpha) \leq K - p\frac{\sigma_0}{2} + \sigma_0.
\end{equation}
Combining \eqref{Main_3} with \eqref{eqn: both cases upper bound} we
can complete the proof:
\begin{equation*}
\Phi(U_n;\Omega_n) \leq \Psi(U_{n-1};\Omega_n;\alpha) - \sigma_0 \leq
K - p\frac{\sigma_0}{2}.
\end{equation*}
\qed
\end{proof}

\section{Lemmas used in the proof Theorem
  \ref{Main.theo1}} \label{sec: lemmas for sigma0 AMLE theorem}

\subsection{The operator $H$ preserves the Lipschitz
  constant} \label{sec: H preserves Lip}

In this section we prove that the sequence of extensions $\{U_n\}_{n
  \in \N}$ constructed in Section \ref{sec: sequence of min
  extensions} are all minimal extensions of the original function $f
\in \F_{\Phi}(X,Z)$. Recall that by construction, $U_0$ is a minimal extension of
$f$, and each $U_n$ is an extension of $f$, so it remains to show that
each $U_n$, for $n \in \N^*$, is minimal. In particular, if we show that
the construction preserves or lowers the Lipschitz constant of the
extension from $U_n$ to $U_{n+1}$ then we are finished. The following
lemma does just that.

\begin{lemma}\label{douanes_lem_0}
Let $F_0 \in \F_{\Phi}(X,Z)$ with $\dom(F_0) = X$ and let $\Omega \in
\OO(\rho,N_0)$. Define $F_1 \in \F_{\Phi}(X,Z)$ as:
\begin{equation*}
F_1(x) \triangleq \left\{
\begin{array}{ll}
H(F_0;\Omega)(x), & \text{if } x \in \Omega, \\
F_0(x), & \text{if } x \in X \setminus \Omega.
\end{array}
\right.
\end{equation*}
Then,
\begin{equation*}
\Phi(F_1;X) \leq \Phi(F_0;X).
\end{equation*}
\end{lemma}

\setcounter{case}{0}

\begin{proof}
We utilize properties \ref{P1} and \ref{P3}. By \ref{P1}, it is
enough to consider the evaluation of $\Phi(F_1;x,y)$ for an
arbitrary pair of points $x,y \in X$. We have three cases:

\begin{case}
If $x,y \in X \setminus \Omega$, then by the
definition of $F_1$ and \ref{P1} (applied to $F_0$) we have:
\begin{equation*}
\Phi(F_1;x,y) = \Phi(F_0;x,y) \leq \Phi(F_0;X).
\end{equation*}
\end{case}

\begin{case}
If $x,y \in \Omega$, then by the definition of
$F_1$, the definition of $H$, and property \ref{P1}, we have:
\begin{equation*}
\Phi(F_1;x,y) = \Phi(H(F_0;\Omega);x,y) \leq
\Phi(F_0;\partial \Omega) \leq \Phi(F_0; X).
\end{equation*}
\end{case}

\begin{case}
Suppose that $x \in X \setminus \Omega$ and $y \in
\Omega$. Assume, for now, that $\overline{B}_{1/2}(x,y) \subset X$. By \ref{P3}
there exists a curve $\gamma \in \Gamma(x,y)$ such that
\begin{equation*}
\Phi(F_1;x,y) \leq \inf_{t \in [0,1]} \max\{
\Phi(F_1;x,\gamma(t)), \Phi(F_1;\gamma(t),y)\}.
\end{equation*}
Let $t_0 \in [0,1]$ be such that $\gamma(t_0) \in \partial
\Omega$. Then, utilizing \ref{P3}, the definition of $F_1$, the
definition of $H$, and \ref{P1}, one has:
\begin{align*}
\Phi(F_1;x,y) &\leq \max\{\Phi(F_1;x,\gamma(t_0)),
\Phi(F_1;\gamma(t_0),y)\} \\
&= \max\{\Phi(F_0;x,\gamma(t_0)),
\Phi(H(F_0;\Omega);\gamma(t_0),y)\} \\
&\leq \Phi(F_0;X).
\end{align*}
If $\overline{B}_{1/2}(x,y) \nsubseteq X$, then we can replace $X$ by a larger
compact set $\widetilde{X} \subset \X$ that does contain
$\overline{B}_{1/2}(x,y)$. By \ref{P2}, extend $F_0$ to a function
$\widetilde{F}_0$ with $\dom(\widetilde{F}_0) = \widetilde{X}$ such
that
\begin{align*}
&\widetilde{F}_0(x) = F_0(x), \quad \text{for all } x \in X, \\
&\Phi(\widetilde{F}_0;\widetilde{X}) = \Phi(F_0;X).
\end{align*}
Define $\widetilde{F}_1$ analogously to $F_1$:
\begin{equation*}
\widetilde{F}_1(x) \triangleq \left\{
\begin{array}{ll}
H(\widetilde{F}_0;\Omega)(x), & \text{if } x \in \Omega, \\
\widetilde{F}_0(x), & \text{if } x \in \widetilde{X} \setminus \Omega.
\end{array}
\right.
\end{equation*}
Note that $\widetilde{F}_1|_X \equiv F_1$, and furthermore, the
analysis just completed at the beginning of case three applies to
$\widetilde{F}_0$, $\widetilde{F}_1$, and $\widetilde{X}$ since
$\overline{B}_{1/2}(x,y) \subset \widetilde{X}$. Therefore,
\begin{equation*}
\Phi(F_1;x,y) = \Phi(\widetilde{F}_1;x,y) \leq
\Phi(\widetilde{F}_0;\widetilde{X}) = \Phi(F_0;X).
\end{equation*}
\end{case}
\qed
\end{proof}

\subsection{Geometrical Lemma}\label{sec: geometrical lemma}

\begin{lemma} \label{lem: geometrical lemma}
Fix $\rho >0$ and $\beta>0$ with $\beta<\rho$. Let
$\{B(x_n;r_n)\}_{n\in \N}$ be a set of balls contained in $X$. Suppose
that $\forall \, n \in \N$,  $r_n> \rho$. For $N \in \N$, let us define 
\begin{equation*}
\A_N \triangleq \bigcup_{n \leq N}B(x_n;r_n) \quad \text{and} \quad \A_{\infty}
\triangleq \bigcup_{n \in \N}B(x_n;r_n).
\end{equation*}
Then $\forall \, \varepsilon >0$, $\exists \, N_{\varepsilon}\in \N$ such
that $\forall \, B(x;r)$, with $r\geq\beta$ and 
$B(x;r) \subset  \A_{\infty}$, we have
\begin{equation*}
B(x;r-\varepsilon) \subset \A_{N_{\varepsilon}}.
\end{equation*}
\end{lemma}

\begin{proof}
Let $\varepsilon >0$.  Let us define for all $N \in \N$,
\begin{equation*}
I_N \triangleq \{ a \mid B(a;\beta-2\varepsilon) \subset \A_N  \} \quad
\text{and} \quad I_{\infty} \triangleq \{b \mid B(b;\beta-2\varepsilon)
\subset \A_{\infty} \}.
\end{equation*}
We remark that $r_n > \rho > \beta-2\varepsilon$ implies that
\begin{equation*}
\A_{N} = \bigcup_{a \in I_N}B(a;\beta-2\varepsilon) \quad \text{and}
\quad \A_{\infty} = \bigcup_{b \in I_{\infty}}B(b;\beta-2\varepsilon).
\end{equation*}
Let us define
\begin{equation*}
\A_N^{\varepsilon/2} \triangleq \bigcup_{a \in
  I_N}B(a;\dfrac{\varepsilon}{2}) \quad \text{and} \quad
\A_{\infty}^{\varepsilon/2} \triangleq \bigcup_{b \in
  I_{\infty}}B(b;\dfrac{\varepsilon}{2}). 
\end{equation*}
The sequence $\{\A_{N}^{\varepsilon/2} \}_{N \in \N}$ is monotonic
under inclusion and converges to $\A_{\infty}^{\varepsilon/2}$ in
Hausdorff distance as $n \rightarrow \infty$. For
$\varepsilon/2>0$ there exists $N_{\varepsilon} \in \N$ such that
\begin{equation}\label{convHaus1}
\delta(\A_{N_{\varepsilon}}^{\varepsilon/2},\A_{\infty}^{\varepsilon/2})\leq \dfrac{\varepsilon}{2}.
\end{equation}
Choose any ball $B(x;r) \subset \A_{\infty}$ with $r \geq \beta$ and define
\begin{equation*}
J(x) \triangleq \{ c \mid B(c;\beta-3\varepsilon) \subset
B(x;r-\varepsilon) \}.
\end{equation*}
We note, similar to earlier, that $r > \beta-2\varepsilon$ implies we have
$B(x;r-\varepsilon) = \bigcup_{c \in J(x)}B(c;\beta-3\varepsilon)$. We will show that
$B(x;r-\varepsilon) \subset \A_{N_{\varepsilon}}$.

Let $y \in B(x;r-\varepsilon)$ and choose $c \in J(x)$ such that $y \in
B(c;\beta-3\varepsilon)$. Since $B(c;\beta-3\varepsilon) \subset
B(x;r-\varepsilon)$ and $B(x;r) \subset \A_{\infty}$ we have
\begin{equation*}
B(c;\beta-2\varepsilon) \subset B(x;r) \subset \A_{\infty}.
\end{equation*}
Thus $c \in I_{\infty}$ and $c \in \A_{\infty}^{\varepsilon/2}$. Since
$c \in \A_{\infty}^{\varepsilon/2}$, using (\ref{convHaus1}), choose $z
\in \A_{N_{\varepsilon}}^{\varepsilon/2}$ which satisfies 
\begin{equation}\label{convHaus2}
d_{\X}(c,z) \leq \dfrac{\varepsilon}{2}.
\end{equation}
Moreover since  $z  \in \A_{N_{\varepsilon}}^{\varepsilon/2}$, choose
$a \in I_{N_{\varepsilon}}$ which satisfies $z \in
B(a;\varepsilon/2)$. We have
\begin{equation}\label{convHaus3}
d_{\X}(z,a) \leq \dfrac{\varepsilon}{2}.
\end{equation}
By \eqref{convHaus2} and \eqref{convHaus3},
\begin{equation}\label{convHaus4}
d_{\X}(c,a) \leq d_{\X}(c,z) + d_{\X}(z,a) \leq
\dfrac{\varepsilon}{2}+\dfrac{\varepsilon}{2} \leq \varepsilon.
\end{equation}
Since $y \in B(c;\beta-3\varepsilon)$ we obtain
\begin{equation}\label{convHaus5}
d_{\X}(y,a) \leq d_{\X}(y,c) + d_{\X}(c,a) \leq \beta-3\varepsilon+\varepsilon \leq
\beta-2\varepsilon.
\end{equation}
Since $a \in I_{N_{\varepsilon}}$ we conclude that $y  \in
\A_{N_{\varepsilon}}$. Therefore $B(x;r-\varepsilon) \subset
\A_{N_{\varepsilon}}$ and the result is proved. \qed
\end{proof}

\subsection{Customs Lemma}\label{sec: customs lemma}

In this section we prove the Customs Lemma, which is vital to the
proof of the property \eqref{eqn: Qp} from Lemma \ref{lem: Qp}. Throughout
this section we shall make use of the construction of the sequence of
extensions $\{U_n\}_{n \in \N}$, which we repeat here.

Let $U_0 \in \F_{\Phi}(X,Z)$ with $\dom(U_0) = X$ and $n \in \N^*$. Set
\begin{equation*}
\Lambda \triangleq \{ \Omega_j \mid \Omega_j \in \OO(\rho,N_0),
\enspace j=1, \ldots, n\},
\end{equation*}
and define:
\begin{equation*}
\A \triangleq \bigcup_{j=1}^n \Omega_j.
\end{equation*}
Let $\{U_j\}_{j=1}^n \subset \F_{\Phi}(X,Z)$ be a collection of functions
defined as:
\begin{equation*}
U_{j+1}(x) \triangleq \left\{
\begin{array}{ll}
H(U_j; \Omega_{j+1})(x), & \text{if } x \in \Omega_{j+1}, \\
U_j(x), & \text{if } x \in X \setminus \Omega_{j+1},
\end{array}
\right. \quad \text{for all } j = 0, \ldots, n-1.
\end{equation*}
We shall need the following lemma first.

\begin{lemma}\label{lem: good zone around x}
Let $x \in \A$. Then there exists $\sigma > 0$ so that $B(x;\sigma)
\subset \A$, and for each $b \in B(x;\sigma)$, there exists $j \in
\{1, \ldots, n\}$ such that $x,b \in \Omega_j$, $U_n(x) = U_j(x)$, and
$U_n(b) = U_j(b)$.
\end{lemma}

\begin{proof}
To begin, set
\begin{equation*}
\eta_1 \triangleq \sup\{r > 0 \mid B(x;r) \subset \A\},
\end{equation*}
noting that $\A$ is open and so $\eta_1 > 0$. Define the following two
sets of indices:
\begin{align*}
I^+ &\triangleq \{j \in \{1, \ldots, n\} \mid x \in \Omega_j\}, \\
I^- &\triangleq \{j \in \{1, \ldots, n\} \mid x \notin \Omega_j\}.
\end{align*}
The set $I^+$ is nonempty since $x \in \A$. So we can additionally define
\begin{equation*}
j^+ \triangleq \max_{j \in I^+} j.
\end{equation*}
On the other hand, $I^-$ may be empty. If it is not, then we define
$\ell_j \triangleq d_{\X}(x,\Omega_j)$ for each $j \in I^-$, and set
\begin{equation*}
\eta_2 \triangleq \frac{1}{2}\min\{\ell_j \mid j \in I^-\}.
\end{equation*}
Finally, we take $\eta$ to be:
\begin{equation*}
\eta \triangleq \left\{
\begin{array}{ll}
\min\{\eta_1,\eta_2\}, & \text{if } I^- \neq \emptyset, \\
\eta_1, & \text{if } I^- = \emptyset.
\end{array}
\right.
\end{equation*}
Note that $\eta > 0$; we also have:
\begin{equation} \label{eqn: eta ball remark 1}
B(x;\eta) \cap \bigcup_{j \in I^-} \Omega_j = \emptyset \qquad
\text{and} \qquad B(x;\eta) \subset \bigcup_{j \in I^+} \Omega_j.
\end{equation}
Now let
\begin{equation*}
J \triangleq \{j \in I^+ \mid U_n(x) = U_j(x)\}.
\end{equation*}
Clearly $j^+ \in J$, and so this set is nonempty. We use it to define
the following: 
\begin{equation*}
\Sigma \triangleq \{b \in B(x;\eta) \mid \exists \thinspace j \in J:
U_n(b) = U_j(b), \enspace b \in \Omega_j\}.
\end{equation*}
The set $\Sigma$ is nonempty since $B(x;\eta) \cap \Omega_{j^+}
\subset \Sigma$. 

To prove the lemma, it is enough to show that $x \in
\mathring{\Sigma}$. Indeed, if $x \in \mathring{\Sigma}$, then there
exists a $\sigma > 0$ such that $B(x;\sigma) \subset
\mathring{\Sigma}$. Then for each $b \in B(x;\sigma)$, there exists $j
\in J$ such that $U_n(b) = U_j(b)$ (by the definition of $\Sigma$) and
$U_n(x) = U_j(x)$ (by the definition of $J$).

We prove that $x \in \mathring{\Sigma}$ by contradiction. Suppose that
$x \notin \mathring{\Sigma}$. Let $\{z_k\}_{k \in \N}$ be a sequence
which converges to $x$ that satisfies the following property: 
\begin{equation*}
\forall \thinspace k \in \N, \enspace z_k \notin \Sigma \text{ and } z_k \in B(x;\eta).
\end{equation*}
Define:
\begin{equation*}
I_k^+ \triangleq \{j \in I^+ \mid z_k \in \Omega_j\}, \quad \text{for
  all } k \in \N. 
\end{equation*}
By the remark given in \eqref{eqn: eta ball remark 1} we see that
$I_k^+$ is nonempty for each $k \in \N$. Thus we can define 
\begin{equation*}
j_k \triangleq \max_{j \in I_k} j.
\end{equation*}
Since $I^+ \setminus J$ has a finite number of elements, there exists
$i_0 \in I^+ \setminus J$ and a subsequence $\{z_{\phi(k)}\}_{k \in
  \N} \subset \{z_k\}_{k \in \N}$ that converges to $x$ such that 
\begin{equation*}
\forall \thinspace k \in \N, \enspace j_{\phi(k)} = i_0.
\end{equation*}
By the definition of $I_k^+$ and using the fact that $i_0$ is the
largest element of $I_{\phi(k)}^+$ for each $k \in \N$, we have 
\begin{equation*}
\forall \thinspace k \in \N, \enspace U_n(z_{\phi(k)}) =
U_{i_0}(z_{\phi(k)}) \text{ and } z_{\phi(k)} \in \Omega_{i_0}. 
\end{equation*}
Since the functions $U_j$ are continuous by \ref{P5}, we have
\begin{equation*}
\lim_{k \rightarrow \infty} U_n(z_{\phi(k)}) = U_n(x),
\end{equation*}
and
\begin{equation*}
\lim_{k \rightarrow \infty} U_{i_0}(z_{\phi(k)}) = U_{i_0}(x).
\end{equation*}
Thus
\begin{equation*}
U_n(x) = U_{i_0}(x).
\end{equation*}
But then $i_0 \in J$, which in turn implies that $z_{\phi(k)} \in
\Sigma$ for all $k \in \N$. Thus we have a contradiction, and so $x
\in \mathring{\Sigma}$. \qed
\end{proof}

\begin{lemma}[Customs Lemma] \label{lem: customs lemma}
If there exists some constant $C \geq 0$ such that
\begin{equation*}
\Phi(U_j; \Omega_j) \leq C, \quad \text{for all } j = 1, \ldots, n,
\end{equation*}
then for all $x,y \in \A$ with $B(x;d_{\X}(x,y))  \subset \A$,
\begin{equation*}
\Phi(U_n;x,y) \leq C.
\end{equation*}
\end{lemma}

\begin{remark}
We call the lemma the ``Customs Lemma'' as it calls to mind traveling
from $x$ to $y$ through the ``countries'' $\{\Omega_j\}_{j=1}^n$. 
\end{remark}

\begin{proof}
Let $x \in \A$ and define
\begin{equation*}
\A(x) \triangleq \{ y \in \A \mid B(x;d_{\X}(x,y))  \subset \A\}.
\end{equation*}
The set $\A(x)$ is a ball centered
at $x$. Furthermore, using Lemma \ref{lem: good zone around x}, there exists a
$\sigma > 0$ and a corresponding ball $B(x;\sigma) \subset \A$ such that
\begin{equation*}
\Phi(U_n;x,b) \leq C, \quad \text{for all } b \in B(x;\sigma).
\end{equation*}
In particular, we have
\begin{equation} \label{eqn: good zone for A(x)}
\Phi(U_n;x,y) \leq C, \quad \text{for all } y \in B(x;\sigma)
\cap \A(x).
\end{equation}
Consider the set
\begin{equation*}
\A_{\sigma}(x) \triangleq \A(x) \setminus (B(x;\sigma)
\cap \A(x)).
\end{equation*}
The set $\A_{\sigma}(x)$ contains those points $y \in \A(x)$
for which we do not yet have an upper bound for
$\Phi(U_n;x,y)$. Let 
\begin{equation*}
M \triangleq \sup_{y \in \A_{\sigma}(x)} \Phi(U_n;x,y).
\end{equation*}
If we can show that $M \leq C$, then we are finished since we took $x$
to be an arbitrary point of $\A$. By \ref{P4}, the function $y \in
\A_{\sigma}(x) \mapsto \Phi(U_n;x,y)$ is continuous. Thus,
\begin{equation*}
M = \sup_{y \in \overline{\A_{\sigma}(x)}} \Phi(U_n;x,y).
\end{equation*}
Since $X$ is compact, $\overline{\A_{\sigma}(x)}$ is compact as well,
and so the set
\begin{equation*}
\mathcal{S} \triangleq \{ y \in \overline{\A_{\sigma}(x)} \mid
\Phi(U_n;x,y) = M \}
\end{equation*}
is nonempty. We select $y_0 \in \mathcal{S}$ such that 
\begin{equation}\label{eqn: def of y0}
d_{\X}(x,y_0) \leq d_{\X}(x,y), \quad \text{for all } y \in \mathcal{S}.
\end{equation}
Since $\mathcal{S}$ is closed and a subset of $\overline{\A_{\sigma}(x)}$, it
is also compact. Furthermore, the function $y \in \mathcal{S} \mapsto d_{\X}(x,y)$ is
continuous, and so the point $y_0$ must exist. It is,
by definition, the point in $\overline{\A_{\sigma}(x)}$ that not only achieves
the maximum value of the function $y \in \overline{\A_{\sigma}(x)} \mapsto
\Phi(U_n;x,y)$, but also, amongst all such points, it is the one
closest to $x$. Thus we have reduced the problem to showing
that $M = \Phi(U_n;x,y_0) \leq C$.

We claim that it is sufficient to show the following: there exists a
point $y_1 \in \A(x)$ such that $d_{\X}(x,y_1) < d_{\X}(x,y_0)$, and furthermore
satisfies:
\begin{equation} \label{eqn: y1 property}
M = \Phi(U_n;x,y_0) \leq \max\{\Phi(U_n;x,y_1), C\}.
\end{equation}
Indeed, if such a point were to exist, then we could complete the
proof in the following way. If $C$ is the max of the right hand side
of \eqref{eqn: y1 property}, then clearly we are finished. If, on the
other hand, $\Phi(U_n;x,y_1)$ is the max, then we have two cases to
consider. If $d_{\X}(x,y_1) < \sigma$, then $y_1 \in B(x;\sigma) \cap
\A(x)$, and so by \eqref{eqn: good zone for A(x)} we know that
$\Phi(U_n;x,y_1) \leq C$. Alternatively, if $d_{\X}(x,y_1) \geq \sigma$,
then $y_1 \in \A_{\sigma}(x)$ and by the definition of $M$ we have
$\Phi(U_n;x,y_1) \leq M$, which by \eqref{eqn: y1 property} implies
that $\Phi(U_n;x,y_1) = M$. But $y_0$ is the closest point to $x$ for
which the function $y \in \overline{\A_{\sigma}(x)} \mapsto \Phi(U_n;x,y)$
achieves the maximum $M$. Thus we have arrived at a contradiction.

Now we are left with the task of showing the existence of such a point
$y_1$. Apply Lemma \ref{lem: good zone around x} to the point $y_0$ to
obtain a radius $\sigma'$ such that $B(y_0;\sigma') \subset \A$ and
for each $b \in B(y_0;\sigma')$, one has $\Phi(U_n;y_0,b) \leq
C$. Since $y_0 \in \A_{\sigma}(x) \subset \A(x)$, we also know that
$B(x;d_{\X}(x,y_0))  \subset \A(x) \subset \A$. Therefore
$\overline{B}_{1/2}(x,y_0) \subset \A(x)$. Let $\gamma: [0,1]
\rightarrow \overline{B}_{1/2}(x,y_0)$ be the curve guaranteed to exist by
\ref{P3}, and take $y_1$ to be the intersection point of $\gamma$ with $\partial
B(y_0;\sigma')$. Clearly $y_1 \in \overline{B}_{1/2}(x,y_0) \subset \A(x)$, and
furthermore it satisfies: 
\begin{align*}
\Phi(U_n;x,y_0) &\leq \inf_{t \in
  [0,1]}\max\{\Phi(U_n;x,\gamma(t)),
\Phi(U_n;\gamma(t),y_0)\} \\ 
&\leq \max\{\Phi(U_n;x,y_1), \Phi(U_n;y_1,y_0)\} \\
&\leq \max\{\Phi(U_n;x,y_1), C\}.
\end{align*}
Finally, using the monotonicity property of the curve
$\gamma$, we see that $d_{\X}(x,y_1) < d_{\X}(x,y_0)$. \qed
\end{proof}

\section{Open questions and future directions}

From here, there are several possible directions. The first was
already mentioned earlier, and involves the behavior of the quasi-AMLE
$U(f,\rho,N_0,\alpha,\sigma_0)$ as $\rho \rightarrow 0$, $N_0
\rightarrow \infty$, $\alpha \rightarrow 0$, and $\sigma_0 \rightarrow
0$. For the limits in $\alpha$ and $\sigma_0$ in particular, it seems
that either more understanding or further exploitation of the
geometrical relationship between $(\X,d_{\X})$ and $(Z,d_Z)$ is
necessary. Should something of this nature be resolved, though, it
would prove the existence of an AMLE under this general setup.

One may also wish to relax the assumptions on $(\X,d_{\X})$. The most
general domain possible in other results is when $(\X,d_{\X})$ is a
length space. Of course then there are far greater restrictions on the
range, and the results only hold for $\Phi = \Lip$. It would seem that
the case of $1$-fields (Section \ref{sec: 1-fields}), in which
$(\X,d_{\X})$ can actually be any Hilbert space, would be a good
specific case in which to work on both this point and the previous
one.

A final possible question concerns the property \ref{P2}. This property
requires that an isometric extension exist for each $f \in
\F_{\Phi}(\X,Z)$; that is, that the Lipschitz constant is preserved
perfectly. What if, however, one had the weaker condition that the
Lipschitz constant be preserved up to some constant? In other words,
suppose that we replace \ref{P2} with the following weaker condition:
\begin{itemize}
\item[$(P_2')$]
Isomorphic Lipschitz extension: \\
For all $f \in \F_{\Phi}(X,Z)$ and for all $D \subset X$ such that
$\dom(f) \subset D$, there exists an extension $F: D \rightarrow Z$
such that
\begin{equation} \label{eqn: P2 isomorphic condition}
\Phi(F;D) \leq C \cdot \Phi(f; \dom(f)),
\end{equation}
where $C$ depends on $(\X,d_{\X})$ and $(Z,d_Z)$.
\end{itemize}
Suppose then we wish to find an $F$ satisfying \eqref{eqn: P2
  isomorphic condition} that also satisfies the AMLE
condition to within a constant factor? The methods here, in which we correct locally, would be
hard to adapt given that with each correction, we would lose a factor
of $C$ in \eqref{eqn: P2 isomorphic condition}. 

\section{Acknowledgements}

E.L.G. is partially supported by the ANR (Agence Nationale de la
Recherche) through HJnet projet ANR-12-BS01-0008-01. M.J.H. would like
to thank IRMAR (The Institue of Research of Mathematics of Rennes) for
supporting his visit in 2011, during which time the authors
laid the foundation for this paper. Both authors would like to
acknowledge the Fields Institute for hosting them for two weeks in
2012, which allowed them to complete this work. 

Both authors would also like to thank the anonymous reviewer for his or her
helpful comments.

\appendix

\section{Equivalence of AMLE definitions} \label{sec: equiv amle defs}

In this appendix we prove that the two definitions for an AMLE with a
generalized functional $\Phi$ are equivalent so long as the domain
$(\X,d_{\X})$ is path connected. First recall the two definitions:
\begin{definition}\label{def: appendix abstract general amle def}
Let $f \in \F_{\Phi}(\X,Z)$ with $\dom(f)$ closed and let $U \in \F_{\Phi}(\X,Z)$ be a
minimal extension of $f$ with $\dom(U) = \X$. Then $U$ is an {\it
  absolutely minimal Lipschitz extension} of $f$ if for every
open set $V \subset \X \setminus \dom(f)$ and every $\widetilde{U} \in
\F_{\Phi}(\X,Z)$ with $\dom(\widetilde{U}) = \X$ that coincides with
$U$ on $\X \setminus V$, 
\begin{equation*}
\Phi(U;V) \leq \Phi(\widetilde{U};V).
\end{equation*}
\end{definition}
\begin{definition}\label{def: appendix aronsson general amle def}
Let $f \in \F_{\Phi}(\X,Z)$ with $\dom(f)$ closed and let $U \in
\F_{\Phi}(\X,Z)$ be a minimal extension
of $f$ with $\dom(U) = \X$. Then $U$ is an {\it absolutely minimal Lipschitz
  extension} of $f$ if
\begin{equation*}
\Phi(U; V) = \Phi(U; \partial V), \quad \text{for all open } V \subset
\X \setminus \dom(f).
\end{equation*}
\end{definition}
\begin{proposition}
Suppose that $(\X,d_{\X})$ is path connected. Then Definition
\ref{def: appendix abstract general amle def} is equivalent to
Definition \ref{def: appendix aronsson general amle def}.
\end{proposition}
\begin{proof}
Since $(\X,d_{\X})$ is path connected, the only sets that are both
open and closed are $\emptyset$ and $\X$. Let $V \subset \X \setminus
\dom(f)$. The case $V = \emptyset$ is vacuous for both definitions,
and since $\dom(f) \neq \emptyset$, the case $V = \X$ is
impossible. Thus every open set $V \subset \X \setminus \dom(f)$ is
not also closed; in particular, $\partial V \neq \emptyset$. 

We first prove that Definition \ref{def: appendix abstract general
  amle def} implies Definition \ref{def: appendix aronsson general
  amle def}. Let $U \in \F_{\Phi}(\X,Z)$ be an AMLE for $f$ satisfying
the condition of Definition \ref{def: appendix abstract general amle
  def}, and suppose by contradiction that $U$ does not satisfy the
condition of Definition \ref{def: appendix aronsson general amle
  def}. That would mean, in particular, that there exists an open set
$V \subset \X \setminus \dom(f)$ such that $\Phi(U; \partial V) <
\Phi(U; V)$. We can then define a new minimal extension $\widetilde{U}
\in \F_{\Phi}(\X,Z)$ as follows:
\begin{equation*}
\widetilde{U}(x) \triangleq \left\{
\begin{array}{ll}
H(U;V)(x), & \text{if } x \in V, \\
U(x), & \text{if } x \in \X \setminus V,
\end{array}
\right.
\end{equation*}
where $H$ is the correction operator defined in Definition
\ref{Def.H}. But then $\widetilde{U}$ coincides with $U$ on $\X
\setminus V$ and $\Phi(\widetilde{U}; V) = \Phi(U; \partial V) <
\Phi(U; V)$, which is a contradiction. 

For the converse, suppose $U \in \F_{\Phi}(\X,Z)$ satisfies Definition
\ref{def: appendix aronsson general amle def} but does not satisfy
Definition \ref{def: appendix abstract general amle def}. Then there
exists and open set $V \subset \X \setminus \dom(f)$ and a function
$\widetilde{U} \in \F_{\Phi}(\X,Z)$ with $\dom(\widetilde{U}) = \X$
that coincides with $U$ on $\X \setminus V$ such that
$\Phi(\widetilde{U}; V) < \Phi(U; V)$. Since $U$ and $\widetilde{U}$
coincide on $\X \setminus V$, $\Phi(\widetilde{U}; \partial V) =
\Phi(U; \partial V)$. On the other hand, $\Phi(\widetilde{U}; \partial
V) \leq \Phi(\widetilde{U};V) < \Phi(U;V) = \Phi(U;\partial V)$. Thus
we have a contradiction. \qed
\end{proof}

\section{Proof that \ref{P0}-\ref{P5} hold for $1$-fields} \label{sec:
  proof of P3 for 1 fields}

In this appendix we consider the case of $1$-fields and the functional
$\Phi = \Gamma^1$ first defined in Section \ref{sec: 1-fields}. Recall
that $(\X,d_{\X}) = \R^d$ with $d_{\X}(x,y) = \|x-y\|$,
where $\|\cdot\|$ is the Euclidean distance. The range $(Z,d_Z)$ is
taken to be $\PP^1(\R^d,\R)$, with elements $P \in \PP^1(\R^d,\R)$
given by $P(a) = p_0 + D_0p \cdot a$, with $p_0 \in \R$, $D_0p \in
\R^d$, and $a \in \R^d$. The distance $d_Z$ is defined as: $d_Z(P,Q)
\triangleq |p_0 - q_0| + \|D_0p - D_0q\|$. For a function $f \in
\F(\X,Z)$, we use the notation $x \in \dom(f) \mapsto f(x)(a) = f_x + D_xf
\cdot (x-a)$, where $f_x \in \R$, $D_xf \in \R^d$, and once again $a
\in \R^d$. Note that $f(x) \in \PP^1(\R^d,\R)$. The functional $\Phi$ is defined as:
\begin{equation} \label{eqn: gamma1 def appendix}
\Phi(f;x,y) = \Gamma^1(f;x,y) \triangleq 2 \sup_{a \in \R^d}
\frac{|f(x)(a) - f(y)(a)|}{\|x-a\|^2 + \|y-a\|^2}.
\end{equation}
Rather than $\Phi$, we shall write $\Gamma^1$ throughout the
appendix. The goal is to show that the properties \ref{P0}-\ref{P5} hold
for $\Gamma^1$ and the metric spaces $(\X,d_{\X})$ and $(Z,d_Z)$. 

\subsection{\ref{P0} and \ref{P1} for $\Gamma^1$}

The property \ref{P0} (symmetry and nonnegative) is clear from the
definition of $\Gamma^1$ in \eqref{eqn: gamma1 def appendix}. The
property \ref{P1} (pointwise evaluation) is by definition.

\subsection{\ref{P2} for $\Gamma^1$}

The property \ref{P2} (existence of a minimal extension to $\X$ for each
$f \in \F_{\Gamma^1}(\X,Z)$) is the main result of
\cite{legruyer:minimalLipschitz1Field2009}. We refer the reader to
that paper for the details.

\subsection{\ref{P3} for $\Gamma^1$}

Showing property \ref{P3}, Chasles' inequality, requires a detailed study of the
domain of uniqueness for a biponctual $1-$field (i.e., when $\dom(f)$
consists of two points). Let $\PP^m(\R^d,\R)$ denote the space of
polynomials of degree $m$ mapping $\R^d$ to $\R$.

For $f \in \F_{\Gamma^1}(\X,Z)$ and $x,y \in \dom(f)$, $x \neq y$ we define
\begin{equation*}
A(f;x,y) \triangleq \frac{2(f_x-f_y)+ (D_xf+D_yf) \cdot (y-x)}{\|x-y \|^2}
\end{equation*}
and
\begin{equation} \label{eqn: B functional}
B(f;x,y) \triangleq \frac{\| D_xf-D_yf\| }{\|x-y \|}.
\end{equation}
Using  \cite{legruyer:minimalLipschitz1Field2009}, Proposition 2.2, we
have  for any $D \subset \dom(f)$,
\begin{equation}\label{AnQ0}
 \Gamma^1(f;D) = \sup_{\substack{x,y \in D \\ x \neq y}} \left(\sqrt{A(f;x,y)^2+B(f;x,y)^2} +|A(f;x,y)|\right). 
\end{equation}

For the remainder of this section, fix $f \in \F_{\Gamma^1}(\X,Z)$, with $\dom(f)
=\{x,y\}$, $x \neq y$, $f(x) = P_x$, $f(y) = P_y$, and set 
\begin{equation}\label{AnQ1}
M \triangleq \Gamma^1 (f;\dom(f)).
\end{equation}
Also, for an arbitrary pair of points $a,b \in \R^d$, let $[a,b]$
denote the closed line segment with end points $a$ and $b$.

\begin{proposition}\label{AnP1}
Let $F$ be an extension of $f$ such that
$\overline{B}_{1/2}(x,y) \subset \dom(F)$. Then there exists a point $c \in
\overline{B}_{1/2}(x,y)$ that depends only on $f$ such that
\begin{equation} \label{eqn: 1-field charles}
\Gamma^1(F;x,y) \leq \max\{\Gamma^1(F;x,a),
\Gamma^1(F;a,y)\}, \text{ for all } a \in [x,c] \cup [c,y].
\end{equation}
\end{proposition}
\begin{remark}
Proposition \ref{AnP1} implies that the operator $\Gamma^1$
satisfies the Chasles' inequality (property \ref{P3}). In particular,
consider an arbitrary $1$-field $g \in \F_{\Gamma^1}(\X,Z)$ with $x,y \in \dom(g)$ such
that $\overline{B}_{1/2}(x,y) \subset \dom(g)$. Then $g$
is trivially an extension of the $1$-field $g|_{\{x,y\}}$, and so in
particular satisfies \eqref{eqn: 1-field charles}. But this is the
Chasles' inequality with $\gamma = [x,c] \cup [c,y]$.
\end{remark}

To prove proposition \ref{AnP1}  we will use the following lemma.
\begin{lemma}\label{AnP2}
There exists $c \in \overline{B}_{1/2}(x,y)$ and $s \in \{-1,1\}$ such that 
\begin{equation*}
M =  2s\dfrac{ P_x(c)- P_y(c) }{\|x-c \|^2 + \|y-c \|^2}.
\end{equation*}
Furthermore,
\begin{align}
&c =  \dfrac{x+y}{2} + s\dfrac{D_xf-D_yf}{2M}, \label{Ape2} \\
&P_x(c)  -s \dfrac{M}{2}\|x-c\|^2   =   P_y(c)+s\dfrac{M}{2}\|y-c\|^2, \nonumber \\
&D_xf+sM (x-c)   =   D_yf-sM (y-c). \nonumber
\end{align}
Moreover, all minimal extensions of $f$ coincide at $c$.
\end{lemma}

The proof of Lemma \ref{AnP2} uses
\cite{legruyer:minimalLipschitz1Field2009}, Propositions 2.2 and
2.13. The details are omitted. Throughout the remainder of this
section, let $c$ denote the point which satisfies Proposition
\ref{AnP2}.

\begin{lemma}\label{AnP3} 
Define $\widetilde{P}_c \in \PP^1(\R^d,\R)$ as
\begin{equation*}
\widetilde{P}_{c}(z) \triangleq \tilde{f}_c + D_c\tilde{f} \cdot
(z-c), \enspace z \in \R^d, 
\end{equation*}
where
\begin{equation*}
 \tilde{f}_c \triangleq  P_x(c)-s\dfrac{M}{2}\|x - c\|^2 ,
\end{equation*}
and
\begin{equation*}
 D_c\tilde{f} \triangleq D_xf+sM(x-c).
\end{equation*}
If $A(f;x,y)=0$, then the following polynomial
\begin{equation*}
 F(z) \triangleq \widetilde{P}_{c}(z)  -s\dfrac{M}{2} \dfrac{[(z-c)
   \cdot (x-c)]^2}{\| x-c \|^2} +s\dfrac{M}{2} \dfrac {[(z-c) \cdot
   (y-c)]^2}{\| y-c\|^2}, \enspace z \in \R^d 
\end{equation*}
is a minimal extension of $f$.

If $A(f;x,y) \neq 0$, let  $z \in \R^d$ and set $p(z) \triangleq (x-c)
\cdot (z-c) $ and $q(z) \triangleq (y-c) \cdot (z-c)$. We define 
\begin{equation*}
F(z) \triangleq \left\{
\begin{array}{ll}
\widetilde{P}_{c}(z) -s\dfrac{M}{2} \dfrac {[(z-c) \cdot (x-c)]^2}{\| x-c \|^2}, &
\text{if }  p(z) \geq 0 \text{ and } q(z) \leq 0, \\
\widetilde{P}_{c}(z) +s\dfrac{M}{2} \dfrac{[(z-c) \cdot (y-c)]^2}{\| y-c \|^2}, &
\text{if } p(z) \leq 0 \text{ and } q(z) \geq 0, \\
\widetilde{P}_{c}(z), & \text{if }  p(z) \leq 0 \text{ and } q(z) \leq
0, \\
\widetilde{P}_{c}(z) -s\dfrac{M}{2} \dfrac{[(z-c) \cdot (x-c)]^2}{\|
  x-c \|^2} +s\dfrac{M}{2} \dfrac{[(z-c) \cdot (y-c)]^2}{\| y-c\|^2},
& \text{if } p(z) \geq 0 \text{ and }  q(z) \geq 0.
\end{array}
\right.
\end{equation*}
Then $F$ is a minimal extension of $f$.
\end{lemma}
\begin{remark}\label{rem: $1$-field from function}
The function $F$ is an extension of the $1$-field $f$ in the following
sense. $F$ defines a $1$-field via its first order Taylor polynomials;
in particular, define the $1$-field $U$ with $\dom(U) = \dom(F)$ as:
\begin{equation*}
U(a) \triangleq J_aF, \enspace a \in \dom(F),
\end{equation*}
where $J_aF$ is the first order Taylor polynomial of $F$. We then
have:
\begin{align*}
&U(x) = f(x) \quad \text{and} \quad U(y) = f(y), \\
&\Gamma^1(U;\dom(U)) = \Gamma^1(f;\dom(f)).
\end{align*}
\end{remark}
\begin{proof}
After showing that the equality $A(f;x,y)=0$ implies that $(x-c) \cdot
(c-y) = 0$, the proof  is easy to check. Suppose that $A(f;x,y)=0$. By
\eqref{eqn: gamma1 def appendix} and \eqref{AnQ1} we have
$M=B(f;x,y)$. By \eqref{Ape2} we have  
\begin{equation*}
\|2c-(x+y)\| = \dfrac{\|D_xf-D_yf\|}{M} = \|x-y\|.
\end{equation*}
Therefore $(x-c) \cdot (c-y) = 0$. \qed
\end{proof}

The proof of the following lemma is also easy to check.
\begin{lemma}\label{AnP5}
Let $g \in \F_{\Gamma^1}(\X,Z)$ such that for all $a \in \dom(g)$, $g(a) = Q_a \in
\PP^1(\R^d,\R)$, with $Q_a(z) = g_a + D_ag \cdot (z-a)$, where
$g_a \in \R$, $D_ag \in \R^d$, and $z \in \R^d$. Suppose there exists
$P \in \mathcal{P}^2(\R^d,\R)$ such that 
\begin{equation*}
P(a) =g_a, \enspace \nabla P(a) = D_ag, \text{ for all }  a \in \dom(g).
\end{equation*}
Then
\begin{equation*}
A(g;a,b) = 0, \text{ for all } a,b \in \dom(g).
\end{equation*}
\end{lemma}
\begin{proof}
Omitted.
\end{proof}

\begin{lemma}\label{AnP4}
All minimal extensions of $f$ coincide on the line segments $[x,c]$ and $[c, y]$.
\end{lemma}
\begin{proof}
First, let $F$ be the minimal extension of $f$ defined in Lemma
\ref{AnP3}, and let $U$ be the $1$-field corresponding to $F$ that was
defined in remark \ref{rem: $1$-field from function}. In particular,
recall that we have:
\begin{equation*}
U(a)(z) = J_aF(z) = F(a) + \nabla F(a) \cdot (z-a), \enspace a \in \dom(F).
\end{equation*}

Now Let $W$ be an arbitrary minimal extension of $f$ such that for all
$a \in \dom(W)$, $W(a) = Q_a \in \PP^1(\R^d,\R)$, with $Q_a(z) = W_a +
D_aW \cdot (z-a)$, where $W_a \in \R$, $D_aW \in \R^d$, and
$z \in \R^d$. We now restrict our attention to $[x,c] \cup [c,y]$. For
any $a \in [x,c] \cup [c,y]$, we write $W(a) = Q_a$ in the following form:
\begin{equation*}
Q_a(z) = F(a) + \nabla F(a) \cdot (z-a) + \delta_a + \Delta_a \cdot
(z-a), \enspace z \in \R^d,
\end{equation*}
where $\delta_a \in \R$ and $\Delta_a \in \R^d$. In particular, we have
\begin{align*}
W_a &= F(a) + \delta_a, \\
D_aW &= \nabla F(a) + \Delta_a.
\end{align*}
Since $U$ is a minimal extension of $f$, it is enough to show that $\delta_a=0$ and 
$ \Delta_a = 0$ for $a \in [x,c] \cup [c,y]$. By symmetry, without
lost generality let us suppose that $a \in [x,c]$.
Since $W$ is a minimal extension of $f$, we have $W_x= F(x) = f_x$,
and by Lemma \ref{AnP2}, $W_c= F(c)$. Using \eqref{AnQ0} and
\eqref{AnQ1}, and once again since $W$ is a minimal extension of $f$, the following
inequality must be satisfied:
\begin{equation}\label{uni1}
|A(W;e,a)| + \dfrac{B(W;e,a)^2}{2M} \leq  \dfrac{M}{2}, \quad e \in \{x,c\}. 
\end{equation}
Using Lemma \ref{AnP5} for $U$ restricted to $\{x,a,c\}$ we have
\begin{equation}\label{uni2}
A(U;e,a)= 0, \quad e \in \{x,c\}. 
\end{equation}
Therefore
\begin{equation}\label{uni22}
A(W;e,a) = \dfrac{|-2\delta_a+ \Delta_a \cdot (e-a)|}{\|e-a\|^2},
\quad e \in \{x,c\}.
\end{equation}

Since  $a \in [x,c]$, we can write $a = c +\alpha(x-c)$ with $\alpha
\in [0,1]$. Using \eqref{uni1} and  \eqref{uni2}, the definition of
$U$, and after simplification, $\delta_a$ and $\Delta_a$ must
satisfy the following inequalities:
\begin{align}
-2\delta_a +  \alpha(1+s) \Delta_a \cdot (c-x)
+\dfrac{\|\Delta_a\|^2}{2M} &\leq 0, \label{uni4_1} \\
2\delta_a +  \alpha(-1+s) \Delta_a \cdot (c-x) +
\dfrac{\|\Delta_a\|^2}{2M} &\leq 0, \label{uni4_2} \\
-2\delta_a - (1- \alpha)(1+s) \Delta_a \cdot (c-x) +
\dfrac{\|\Delta_a\|^2}{2M} &\leq 0, \label{uni5_1} \\
2\delta_a- (1- \alpha)(-1+s) \Delta_a \cdot (c-x) +
\dfrac{\|\Delta_a\|^2}{2M} &\leq 0. \label{uni5_2}
\end{align} 
The  inequality $(1-\alpha)( \eqref{uni4_1}+ \eqref{uni4_2} )+ \alpha(
\eqref{uni5_1}+ \eqref{uni5_2} )$  implies that $\Delta_a
=0$. Furthermore, the inequalities \eqref{uni4_1} and \eqref{uni4_2}
imply that $\delta_a =0$. Now the proof is  complete. \qed
\end{proof}

We finish this appendix by proving Proposition \ref{AnP1}. Let us use
the notations of Proposition \ref{AnP1} where $c$ satisfies Lemma
\ref{AnP2}. By Lemma \ref{AnP4}, the extension $U$ (defined in Remark
\ref{rem: $1$-field from function}) of $f$ is the
unique minimal extension of $f$ on the restriction to $[x,c] \cup
[c,y]$. Moreover, we can check that 
\begin{equation}\label{uni6}
\Gamma^1(f;x,y) = \max\{\Gamma^1(U;x,a),
\Gamma^1(U;a,y)\}, \text{ for all } a \in [x,c] \cup [c,y].
\end{equation}
Let $W$ be an extension of $f$. By contradiction suppose that there
exists $a \in [x,c] \cup [c,y]$ such that
\begin{equation}\label{uni7}
\Gamma^1(f;x,y) > \max\{\Gamma^1(W;x,a), \Gamma^1(W;a,y)\}.
\end{equation}
Using  \cite{legruyer:minimalLipschitz1Field2009}, Theorem 2.6, for
the $1$-field $g \triangleq \{f(x),W(a),f(y)\}$ of domain $\{x,a,y\}$
there exists an extension $G$ of $g$ such that
\begin{equation}\label{uni8}
\Gamma^1(G;\dom(G)) \leq  \Gamma^1(f;x,y).
\end{equation}
Therefore $G$ is a minimal extension of $f$. By Lemma \eqref{AnP4} and
the definition of $G$ we have  $W(a) = G(a) = U(a)$. But then by
\eqref{uni6},\eqref{uni7}, and \eqref{uni8} we obtain a
contradiction. Now the proof of the Proposition \ref{AnP1} is
complete. \qed

\subsection{\ref{P4} for $\Gamma^1$}

Property \ref{P4} (continuity of $\Gamma^1$) can be shown using
\eqref{AnQ0}, and a series of elementary calculations. We omit the details.

\subsection{\ref{P5} for $\Gamma^1$}

To show property \ref{P5} (continuity of $f \in \F_{\Gamma^1}(\X,Z)$), we
first recall the definition of $d_Z$. For $P \in \PP^1(\R^d,\R)$ with
$P(a) = p_0 + D_0p \cdot a$, $p_0 \in \R$, $D_0p \in \R^d$, we have
\begin{equation*}
d_Z(P,Q) = |p_0 - q_0| + \|D_0p - D_0q\|.
\end{equation*}
Recall also that for a $1$-field $f: E \rightarrow Z$, $E \subset \X$,
we have:
\begin{equation*}
x \in E \mapsto f(x)(a) = f_x + D_xf \cdot (a-x) = (f_x - D_xf \cdot
x) + D_xf \cdot a, \quad a \in \R^d.
\end{equation*}

To show continuity of $f \in \F_{\Gamma^1}(\X,Z)$ at $x \in \X$, we need the
following: for all $\varepsilon > 0$, there exists a $\delta > 0$ such
that if $\|x-y\| < \delta$, then $d_Z(f(x),f(y)) <
\varepsilon$. Consider the following:
\begin{align}
d_Z(f(x),f(y)) &= |f_x - D_xf \cdot x - f_y + D_yf \cdot y| + \|D_xf -
D_yf\| \nonumber \\
&\leq |f_x - f_y| + |D_xf \cdot x - D_yf \cdot y| + \|D_xf -
D_yf\|. \label{eqn: dR for P5 3 parts}
\end{align}
We handle the three terms \eqref{eqn: dR for P5 3 parts} separately
and in reverse order.

For the third term, recall the definition of $B(f;x,y)$ in \eqref{eqn:
  B functional}, and define $B(f;E)$ accordingly; we then have:
\begin{equation} \label{eqn: application of B}
\|D_xf - D_yf\| \leq B(f;E) \|x-y\| \leq \Gamma^1(f;E) \|x-y\|.
\end{equation}
Since $\Gamma^1(f;E) < \infty$, that completes this term.

For the second term:
\begin{align*}
|D_xf \cdot x - D_yf \cdot y| &\leq |D_xf \cdot (x-y)| + |(D_xf -
D_yf) \cdot y| \\
&\leq \|D_xf\| \|x-y\| + \|D_xf - D_yf\| \|y\|
\end{align*}
Using \eqref{eqn: application of B}, we see that this term can be
made arbitrarily small using $\|x-y\|$ as well.

For the first term $|f_x - f_y|$, define $g: E \rightarrow \R$ as
$g(x) = f_x$ for all $x \in E$. By Proposition 2.5 of
\cite{legruyer:minimalLipschitz1Field2009}, the function $g$ is
continuous. This completes the proof. \qed

\section{Proof of Proposition \ref{prop: alpha to zero}} \label{sec:
  proof of alpha to zero}

We prove Proposition \ref{prop: alpha to zero}, which we restate here:

\begin{proposition}[Proposition \ref{prop: alpha to zero}]
Let $f \in \F_{\Phi}(\X,Z)$. For any open $V \subset \dom(f)$, $V \neq \X$, and $\alpha \geq 0$, let us define
\begin{equation*}
V_{\alpha}  \triangleq \{ x \in V \mid d_{\X}(x,\partial V) \geq \alpha \}.
\end{equation*}
Then for all $ \alpha > 0$,
\begin{equation}\label{Psi1app}
\Phi(f;V_{\alpha}) \leq \max \{\Psi(f;V;\alpha), \Phi(u;\partial V_{\alpha})\},
\end{equation}
and
\begin{equation}\label{Psi2app}
\Phi(f;V) = \max \{\Psi(f;V;0), \Phi(f;\partial V)\}.
\end{equation}
\end{proposition}

\setcounter{case}{0}

\begin{proof}
For the first statement fix $\alpha >0$ and an open set $V \subset
\X$, $V \neq \X$. For
proving (\ref{Psi1app}), it is  sufficient to prove that for all $x
\in \mathring{V}_{\alpha}$ and for all $y \in V_{\alpha}$ we have
\begin{equation}\label{Psi3}
\Phi(f;x,y) \leq \max\{\Psi(f;V;\alpha), \Phi(f;\partial V_{\alpha})\}.
\end{equation}

Fix  $x \in \mathring{V}_{\alpha}$. Let  $B(x;r_{x}) \subset V$ be a ball such 
that $r_{x}$ is maximized and define
\begin{equation*}
M(x) \triangleq \sup \left\{ \Phi(f;x,y) \mid  y \in \overline{V_{\alpha} 
\setminus B(x;r_{x})} \right\},
\end{equation*}
as well as
\begin{equation*}
\Delta(x) \triangleq \left\{ y \in \overline{V_{\alpha} \setminus B(x;r_{x})} \mid
\Phi(f;x,y) = M(x) \right\},
\end{equation*}
and
\begin{equation*}
\delta(x) \triangleq \inf \{d_{\X}(x,y) \mid y\in \Delta(x) \}.
\end{equation*}
We have three cases:
\begin{case}
Suppose $M(x) \leq \sup\{\Phi(f;x,y) \mid y \in B(x;r_{x}) \}$. Since
$B(x;r_{x}) \subset V$ with $r_{x} \geq \alpha$ we have 
\begin{equation*}
\Phi(f;x,y)  \leq \Psi(f;V;\alpha), \enspace \forall \, y \in B(x;r_{x}).
\end{equation*}
Therefore $M(x) \leq \Psi(f;V;\alpha)$. That completes the first case.
\end{case}

For cases two and three, assume that $M(x) > \sup\{ \Phi(f;x,y) \mid y
\in B(x;r_x)\}$ and select $y \in \Delta(x)$ with $d_{\X}(x,y) = \delta(x)$.
\begin{case}
Suppose $y \in \mathrm{int}(V_{\alpha}\setminus B(x;r_{x}))$. Let  
$B(y;r_y)\subset V$ be a ball such that $r_y$ is maximal.
Consider the curve $\gamma \in \Gamma(x,y)$ satisfying \ref{P3}. Let
$m \in \gamma \cap B(y;r_y) \cap V_{\alpha}$, $m \neq x,y$. Using \ref{P3}, we have
\begin{equation}\label{Psi4}
\Phi(f;x,y) \leq \max\{\Phi(f;x,m), \Phi(f;m,y)\}.
\end{equation}
Using the monotonicity of $\gamma$ we have $d_{\X}(x,m) < d_{\X}(x,y)$. Using the 
minimality of the distance of $d_{\X}(x,y)$ and since $m \in V_{\alpha}$ we have 
$\Phi(f;x,m) < \Phi(f;x,y)$. Therefore
\begin{equation}\label{Psi5}
\Phi(f;x,y) \leq \Phi(f;m,y).
\end{equation}
Since $m \in B(y;r_y)$ with $r_y \geq \alpha$, using the definition of $\Psi$ we have 
$\Phi(f;m,y) \leq \Psi(f;V;\alpha)$. Therefore $M(x) \leq \Psi(f;V;\alpha)$.
\end{case}

\begin{case}
Suppose $y \in \partial V_{\alpha}\setminus B(x;r_{x})$. As in case two,  let  
$B(y;r_y)\subset V$ be a ball such that $r_y$ is maximal and consider the curve 
$\gamma \in \Gamma(x,y)$ satisfying \ref{P3}. Let $m \in \gamma \cap
B(y;r_y) \cap V_{\alpha}$. If there exists $m \neq y$ in $V_{\alpha}$,
we can apply the same reasoning as in case two and we have $M(x) \leq
\Psi(f;V;\alpha)$.

If $m = y$ is the only element of $\gamma \cap B(y;r_y) \cap
V_{\alpha}$, then there still exists $m' \in \gamma \cap \partial
V_{\alpha}$ with $m' \neq y$. Using \ref{P3} we have
\begin{equation}\label{Psi6}
 \Phi(f;x,y) \leq \max\{\Phi(f;x,m'), \Phi(f;m',y)\}.
\end{equation}
Using the monotonicity of $\gamma$ we have $d_{\X}(x,m') < d_{\X}(x,y)$. Using the 
minimality of distance of $d_{\X}(x,y)$ and since $m' \in V_{\alpha}$ we have 
$\Phi(f;x,m') < \Phi(f;x,y)$. Therefore
\begin{equation}\label{Psi7}
\Phi(f;x,y) \leq \Phi(f;m',y).
\end{equation}
Since $m',y \in \partial V_{\alpha}$, we obtain the following majoration
\begin{equation}\label{Psi8}
\Phi(f;x,y) \leq \Phi(f;m',y) \leq \Phi(f;\partial V_{\alpha}),
\end{equation}
which in turn gives:
\begin{equation*}
M(x) \leq \Phi(f;\partial V_{\alpha}).
\end{equation*}
The inequality (\ref{Psi1app}) is thus demonstrated.
\end{case}

For the second statement, we note that by the definition of $\Psi$ we have
\begin{equation*}
\max \{\Psi(f;V;0), \Phi(f;\partial V)\}\leq \Phi(f;V).
\end{equation*}
Using \eqref{Psi1app}, to show \eqref{Psi2app} it is sufficient to prove
\begin{equation}\label{Psi9}
\lim_{\alpha\rightarrow 0}\Phi(f;V_{\alpha}) =\Phi(f;V).
\end{equation}
Let $\varepsilon>0$. Then there exists $x_{\varepsilon} \in V$ and $y_{\varepsilon} 
\in \overline{V}$ such that
\begin{equation*}
 \Phi(f;V) \leq \Phi(f;x_{\varepsilon},y_{\varepsilon})+ \varepsilon.
\end{equation*}
Set $r_{\varepsilon}=d_{\X}(x_{\varepsilon},\partial V)$. If $y_{\varepsilon} \in
V$, there exists $\tau_1$ with $0 < \tau_1 \leq
r_{\varepsilon}$ such that for all $\alpha$, $0 < \alpha \leq \tau_1$,
$(x_{\varepsilon},y_{\varepsilon}) \in V_{\alpha}\times V_{\alpha}$. Therefore
\begin{equation*}
\Phi(f;x_{\varepsilon},y_{\varepsilon}) \leq \Phi(f;V_{\alpha}), \quad
\forall \, \alpha, \enspace 0 < \alpha \leq \tau_1.
\end{equation*}
If, on the other hand, $y_{\varepsilon} \in \partial V$, using \ref{P4}
there exists $\tau_2$  with $0 < \tau_2 \leq
\min\{r_{\varepsilon},\tau_1\}$, such that
\begin{equation*}
|\Phi(f;x_{\varepsilon},m)-\Phi(f;x_{\varepsilon},y_{\varepsilon})|
\leq \varepsilon, \quad \forall \, m \in B(y_{\varepsilon}; \tau_2).
\end{equation*}
By choosing $m \in  B(y_{\varepsilon};\tau_2) \cap V_{\tau_2}$, we obtain
\begin{equation*}
\Phi(f;x_{\varepsilon},y_{\varepsilon})\leq
\Phi(f;V_{\alpha})+\varepsilon, \quad \forall \, \alpha, \enspace 0<
\alpha \leq \tau_2.
\end{equation*}
Therefore $\Phi(f;V) \leq \Phi(f;V_{\alpha}) + 2\varepsilon$, for all
$\alpha$ such that $0< \alpha \leq \tau_2$ and for all $\varepsilon >
0$. Thus \eqref{Psi9} is true. \qed
\end{proof}

\bibliography{AMLEbibliography}

\begin{thebibliography}{10}
\providecommand{\url}[1]{{#1}}
\providecommand{\urlprefix}{URL }
\expandafter\ifx\csname urlstyle\endcsname\relax
  \providecommand{\doi}[1]{DOI~\discretionary{}{}{}#1}\else
  \providecommand{\doi}{DOI~\discretionary{}{}{}\begingroup
  \urlstyle{rm}\Url}\fi

\bibitem{armstrong:easyProofJensen2010}
Armstrong, S.N., Smart, C.K.: As easy proof of {Jensen's} theorem on the
  uniqueness of infinity harmonic functions.
\newblock Calculus of Variations and Partial Differential Equations
  \textbf{37}(3-4), 381--384 (2010)

\bibitem{aronsson:minSupFI1965}
Aronsson, G.: Minimization problems for the functional $\sup_xf(x,f(x),f'(x))$.
\newblock Arkiv f\"{o}r Matematik \textbf{6}, 33--53 (1965)

\bibitem{aronsson:minSupFII1966}
Aronsson, G.: Minimization problems for the functional $\sup_xf(x,f(x),f'(x))$
  \text{II}.
\newblock Arkiv f\"{o}r Matematik \textbf{6}, 409--431 (1966)

\bibitem{aronsson:AMLE1967}
Aronsson, G.: Extension of functions satisfying \text{Lipschitz} conditions.
\newblock Arkiv f\"{o}r Matematik \textbf{6}, 551--561 (1967)

\bibitem{aronsson:tourAMLE2004}
Aronsson, G., Crandall, M.G., Juutinen, P.: A tour of the theory of absolutely
  minimizing functions.
\newblock Bulletin of the American Mathematical Society \textbf{41}, 439--505
  (2004)

\bibitem{barles:existNonLinearElliptic2001}
Barles, G., Busca, J.: Existence and comparison results for fully nonlinear
  degenerate elliptic equations.
\newblock Communications in Partial Differential Equations \textbf{26}(11--12),
  2323--2337 (2001)

\bibitem{jensen:uniqueAMLE1993}
Jensen, R.: Uniqueness of \text{L}ipschitz extensions: \text{M}inimizing the
  sup norm of the gradient.
\newblock Archive for Rational Mechanics and Analysis \textbf{123}(1), 51--74
  (1993)

\bibitem{juutinen:amleMetricSpace2002}
Juutinen, P.: Absolutely minimizing \text{L}ipschitz extensions on a metric
  space.
\newblock Annales Academiae Scientiarum Fennicae Mathematica \textbf{27},
  57--67 (2002)

\bibitem{kelley:banachExtension1952}
Kelley, J.L.: Banach spaces with the extension property.
\newblock Transactions of the American Mathematical Society \textbf{72}(2),
  323--326 (1952)

\bibitem{kirszbraun:lipschitzTransformations1934}
Kirszbraun, M.D.: \"{U}ber die zusammenziehende und \text{L}ipschitzsche
  \text{T}ransformationen.
\newblock Fundamenta Mathematicae \textbf{22}, 77--108 (1934)

\bibitem{mcshane:extensionRangeFcns1934}
McShane, E.J.: Extension of range of functions.
\newblock Bulletin of the American Mathematical Society \textbf{40}(12),
  837--842 (1934)

\bibitem{milman:absMinExtMetric1999}
Milman, V.A.: Absolutely minimal extensions of functions on metric spaces.
\newblock Matematicheskii Sbornik \textbf{190}(6), 83--110 (1999)

\bibitem{nachbin:hahnBanachLinearTrans1950}
Nachbin, L.: A theorem of the \text{H}ahn-\text{B}anach type for linear
  transformations.
\newblock Transactions of the American Mathematical Society \textbf{68}(1),
  28--46 (1950)

\bibitem{naor:treeAMLE2012}
Naor, A., Sheffield, S.: Absolutely minimal \text{L}ipschitz extension of
  tree-valued mappings.
\newblock Mathematische Annalen \textbf{354}(3), 1049--1078 (2012)

\bibitem{peres:tugOfWar2009}
Peres, Y., Schramm, O., Sheffield, S., Wilson, D.B.: Tug-of-war and the
  infinity {Laplacian}.
\newblock Journal of the American Mathematical Society \textbf{22}(1), 167--210
  (2009)

\bibitem{sheffield:vectorAMLE2012}
Sheffield, S., Smart, C.K.: Vector-valued optimal \text{L}ipschitz extensions.
\newblock Communications on Pure and Applied Mathematics \textbf{65}(1),
  128--154 (2012)

\bibitem{legruyer:amlePDE2007}
\text{Le Gruyer}, E.: On absolutely minimizing \text{L}ipschitz extensions and
  \text{PDE} {$\Delta_{\infty}(u) = 0$}.
\newblock NoDEA: Nonlinear Differential Equations and Applications
  \textbf{14}(1-2), 29--55 (2007)

\bibitem{legruyer:minimalLipschitz1Field2009}
\text{Le Gruyer}, E.: Minimal \text{L}ipschitz extensions to differentiable
  functions defined on a \text{H}ilbert space.
\newblock Geometric and Functional Analysis \textbf{19}(4), 1101--1118 (2009)

\bibitem{valentine:vectorLipschitzExt1945}
Valentine, F.A.: A \text{L}ipschitz condition preserving extension for a vector
  function.
\newblock American Journal of Mathematics \textbf{67}(1), 83--93 (1945)

\bibitem{wells:diffFcnsLipDer1973}
Wells, J.C.: Differentiable functions on \text{Banach} spaces with
  \text{Lipschitz} derivatives.
\newblock Journal of Differential Geometry \textbf{8}, 135--152 (1973)

\bibitem{wells:embeddingExtensions1975}
Wells, J.H., Williams, L.R.: Embeddings and Extensions in Analysis.
\newblock Springer-Verlag, New York Heidelberg Berlin (1975)

\bibitem{whitney:analyticExtensions1934}
Whitney, H.: Analytic extensions of differentiable functions defined in closed
  sets.
\newblock Transactions of the American Mathematical Society \textbf{36}(1),
  63--89 (1934)

\end{thebibliography}

 \end{document}